\documentclass[12pt]{amsart}
\usepackage[utf8]{inputenc}
\usepackage{mathrsfs} 
\usepackage{amssymb}
\usepackage{color}
\usepackage{bm}
\usepackage{graphicx}

\usepackage[centertags]{amsmath}
\usepackage{amsfonts}
\usepackage{amsthm}
\usepackage{enumerate}
\usepackage{tocvsec2}
\usepackage{xcolor}
\usepackage[margin=.75in]{geometry}

\newtheorem{theorem}{Theorem}[section]
\newtheorem*{theorem*}{Theorem}

\newtheorem{corollary}[theorem]{Corollary}
\newtheorem{lemma}[theorem]{Lemma}
\newtheorem{rem}[theorem]{Remark}

\newtheorem{proposition}[theorem]{Proposition}

\newtheorem{example}{Example}[section]

\theoremstyle{definition}
\newtheorem{definition}[theorem]{Definition}

\newcommand{\nn}{\mathbb{N}}
\newcommand{\ee}{\varepsilon}

\newcommand{\upp}{\upharpoonright}

\begin{document}

\title[Uniform subsequential estimates]{Uniform subsequential estimates on weakly null sequences}

\author{M. Brixey}
\address{Department of Mathematics, Miami University, Oxford, OH 45056, USA}
\email{brixeym@miamioh.edu}

\author{R.M. Causey}
\address{Department of Mathematics, Miami University, Oxford, OH 45056, USA}
\email{causeyrm@miamioh.edu}

\author{P. Frankart}
\address{Department of Mathematics, Miami University, Oxford, OH 45056, USA}
\email{frankapa@miamioh.edu}

\begin{abstract} We provide a generalization of two results of Knaust and Odell from \cite{KO2} and \cite{KO}. We prove that if $X$ is a Banach space and $(g_n)_{n=1}^\infty$ is a right dominant Schauder basis such that every normalized, weakly null sequence in $X$ admits a subsequence dominated by a subsequence of $(g_n)_{n=1}^\infty$, then there exists a constant $C$ such that every normalized, weakly null sequence in $X$ admits a subsequence $C$-dominated by a subsequence of $(g_n)_{n=1}^\infty$. We also prove that if every spreading model generated by a normalized, weakly null sequence in $X$ is dominated by some spreading model generated by a subsequence of $(g_n)_{n=1}^\infty$, then there exists $C$ such that every spreading model generated by a normalized, weakly null sequence in $X$ is $C$-dominated by every spreading model generated by a subsequence of $(g_n)_{n=1}^\infty$. We also prove a single, ordinal-quantified result which unifies and interpolates between these two results. 

\end{abstract}


\maketitle

\section{Introduction}

The Principle of Uniform Boundedness is one of the cornerstone theorems of Banach space theory.  When they are available, uniform estimates are highly desirable.  In certain situations, the prototypical example of which is the Principle of Uniform Boundedness,   a certain, non-uniform bound automatically implies a uniform one. Several instances of such a phenomenon have appeared (\cite{KO2,KO, F}), in which  the estimates take the form of domination of subsequences of normalized, weakly null sequences by a fixed basis Let us recall that for sequences $(e_n)_{n=1}^\infty$, $(f_n)_{n=1}^\infty$ in possibly different Banach spaces, we say that  $(f_n)_{n=1}^\infty$ $C$-\emph{dominates} $(e_n)_{n=1}^\infty$ if for any finitely supported scalar sequence $(a_n)_{n=1}^\infty$, $$\|\sum_{n=1}^\infty a_ne_n\|\leqslant C\|\sum_{n=1}^\infty a_nf_n\|.$$ We say that $(f_n)_{n=1}^\infty$ \emph{dominates} $(e_n)_{n=1}^\infty$ if $(f_n)_{n=1}^\infty$ $C$-dominates $(e_n)_{n=1}^\infty$ for some $C$. 

The following result was shown for $c_0$ in \cite{KO2} and for $\ell_p$, $1<p<\infty$, in  \cite{KO}. 

\begin{theorem} Fix $1<p<\infty$. If $X$ is a Banach space in which every normalized, weakly null sequence admits a subsequence dominated by the $\ell_p$ basis, then there exists a constant $C$ such that every normalized, weakly null sequence in $X$ has a subsequence $C$-dominated by the $\ell_p$ basis. The analogous result holds if $\ell_p$ is replaced by $c_0$. 

\label{KO}
\end{theorem}

Freeman gave the following generalization.

\begin{theorem} If $(v_n)_{n=1}^\infty$ is any seminormalized Schauder basis and $X$ is some Banach space such that every normalized, weakly null sequence in $X$ admits a subsequence dominated by $(v_n)_{n=1}^\infty$, then there exists a constant $C$ such that every normalized, weakly null sequence in $X$ admits a subsequence $C$-dominated by $(v_n)_{n=1}^\infty$. 

\label{Freeman}
\end{theorem}

We recall that a Schauder basis $(g_n)_{n=1}^\infty$ is \emph{right dominant} if for any two sequences $m_1<m_2<\ldots$, $l_1<l_2<\ldots$ of positive integers such that $m_n\leqslant l_n$ for all $n\in\nn$, $(g_{m_n})_{n=1}^\infty$ is dominated by $(g_{l_n})_{n=1}^\infty$.  For $1\leqslant r<\infty$, we say $(g_n)_{n=1}^\infty$ is $r$-\emph{right dominant} if for any two sequences $m_1<m_2<\ldots$, $l_1<l_2<\ldots$ of positive integers such that $m_n\leqslant l_n$ for all $n\in\nn$, $(g_{m_n})_{n=1}^\infty$ is $r$-dominated by $(g_{l_n})_{n=1}^\infty$.  A standard gliding hump argument implies that a right dominant basis is $r$-right dominant for some $1\leqslant r<\infty$.

For a fixed seminormalized Schauder basis $(v_n)_{n=1}^\infty$, and $C>0$, we say a Banach space $X$ has property  $U_{(v_n)_{n=1}^\infty}$ (resp. $C$-$U_{(v_n)_{n=1}^\infty}$) if every  weakly null sequence in $B_X$ has a subsequence dominated (resp. $C$-dominated) by $(v_n)_{n=1}^\infty$. Then Freeman's theorem can be summarized as saying that a Banach space has $U_{(v_n)_{n=1}^\infty}$ if and only if it has $C$-$U_{(v_n)_{n=1}^\infty}$ for some $C>0$.  We wish to consider the following property, which for a fixed basis is  weaker than that studied by Freeman. Given a fixed seminormalized, right dominant Schauder basis $(g_n)_{n=1}^\infty$, a Banach space $X$, and $C>0$, we say $X$ has $S_{(g_n)_{n=1}^\infty}$ (resp. $C$-$S_{(g_n)_{n=1}^\infty}$) if every weakly null sequence in $B_X$ admits a subsequence dominated (resp. $C$-dominated) by a subsequence of $(g_n)_{n=1}^\infty$.  It is easy to see that $C$-$U_{(g_n)_{n=1}^\infty}\Rightarrow C$-$S_{(g_n)_{n=1}^\infty}$ and $U_{(g_n)_{n=1}^\infty}\Rightarrow S_{(g_n)_{n=1}^\infty}$.  The first part of this paper is devoted to proving an analogue of Freeman's theorem regarding the properties $S_{(g_n)_{n=1}^\infty}$ and $C$-$S_{(g_n)_{n=1}^\infty}$ in place of $U_{(g_n)_{n=1}^\infty}$ and $C$-$U_{(g_n)_{n=1}^\infty}$, and an ordinal generalization thereof.  We devote the final section of this paper to examples which show that the reverse implications do not hold. That is, $S_{(g_n)_{n=1}^\infty}\not\Rightarrow U_{(g_n)_{n=1}^\infty}$.  More generally, we give examples of $1$-right dominant bases $(g_n)_{n=1}^\infty$ and Banach spaces $X$ such that $X$ has $S_{(g_n)_{n=1}^\infty}$, but if $X$ has $U_{(v_n)_{n=1}^\infty}$, then $(v_n)_{n=1}^\infty$ must be equivalent to the $\ell_1$ basis.

We will find it convenient to work in more generality than with normalized, weakly null sequences. Throughout, given a Banach space $X$, we will deal with a (not necessarily closed) subspace $R$ of $\ell_\infty(X)$ such that $c_{00}(X)\subset R$ and $R$ is endowed with some norm $\|\cdot\|_R$ such that, with $B_R$ denoting the unit ball of $R$ with respect to the norm $\|\cdot\|_R$,  \begin{enumerate}[(i)]\item $(R, \|\cdot\|_R)$ is a Banach space, \item any subsequence of a member of $B_R$ is a member of $B_R$, \item $B_R\subset B_{\ell_\infty(X)}$.\end{enumerate} We will call such a space $(R, \|\cdot\|_R)$ a \emph{subsequential space on} $X$.

Our main result is the following.

\begin{theorem} Let $(g_n)_{n=1}^\infty$ be a seminormalized, right dominant Schauder basis. Let $X$ be a Banach space and let $(R, \|\cdot\|_R)$ be a subsequential space on $X$.  If every member of $R$ admits a subsequence dominated by a subsequence of $(g_n)_{n=1}^\infty$, then there exists $C>0$ such that every member of $B_R$ admits a subsequence $C$-dominated by a subsequence of $(g_n)_{n=1}^\infty$. 

\label{main}
\end{theorem}

If we let $R=c_0^w(X)$, the space of weakly null sequences in $X$,  and $\|\cdot\|_R=\|\cdot\|_{\ell_\infty(X)}$, then by homogeneity, Theorem 
\ref{main} implies the following.

\begin{corollary} Let $(g_n)_{n=1}^\infty$ be a seminormalized, right dominant Schauder basis. If every normalized, weakly null sequence in $X$ has a subsequence dominated by a subsequence of $(g_n)_{n=1}^\infty$, then there exists $C>0$ such that every normalized, weakly null sequence in $X$ admits a subsequence $C$-dominated by a subsequence of $(g_n)_{n=1}^\infty$. 

\label{main5}
\end{corollary}

Knaust and Odell also gave a proof of the following theorem in \cite{KO}.

\begin{theorem} Let $X$ be a Banach space in which every spreading model generated by a normalized, weakly null sequence in $X$ is dominated by the canonical $\ell_p$ basis. Then there exists $C$ such that every spreading model generated by a normalized, weakly null sequence in $X$ is $C$-dominated by the $\ell_p$ basis.

\label{fold}

\end{theorem}

We provide the following generalization.

\begin{theorem} Let $(g_n)_{n=1}^\infty$ be a seminormalized, right dominant basis. Let $X$ be a Banach space and let $(R, \|\cdot\|_R)$ be a subsequential space on $X$. If every spreading model generated by a member of $R$ is dominated by a spreading model generated by a subsequence of $(g_n)_{n=1}^\infty$, then there exists a constant $C$ such that every spreading model generated by a member of $B_R$ is $C$-dominated by every spreading model generated by a subsequence of $(g_n)_{n=1}^\infty$.

\label{ffold}
\end{theorem}

Resembling one of the proofs of the Principle of Uniform Boundedness, the proofs of Theorem \ref{KO} by Knaust and Odell and Theorem \ref{Freeman} of Freeman, are obtained by assuming the sequence/subseqeuence hypothesis holds, but not uniformly. Then one obtains an array $(x^k_n)_{n,k=1}^\infty$ such that for each $k$, $(x^k_n)_{n=1}^\infty$ is a normalized, weakly null sequence which is $C_k$-dominated by $(v_n)_{n=1}^\infty$ (where $(v_n)_{n=1}^\infty$ is the canonical basis of $\ell_p$ or $c_0$ in \cite{KO}, and is an arbitrary seminormalized Schauder basis in \cite{F}), but such that row $k$ is not dominated by $(v_n)_{n=1}^\infty$ with any constant better than $D_k$. Then one combines the sequences with absolutely converging weights $w_k$.  One has to know that the combined sequence does not admit any subsequence dominated by $(v_n)_{n=1}^\infty$. Of course, this is done by showing that since for an arbitrary $k$, $(x^k_n)_{n=1}^\infty$ is not $D_k$-dominated by $(v_n)_{n=1}^\infty$, then a careful choice of $D_k, C_k$, and $w_k$, the combined sequence is not $f(D_k)$-dominated by $(v_n)_{n=1}^\infty$, where $f$ is some function such that $\lim_{x\to \infty}f(x)=\infty$.    The difficulty is knowing the badness of $(x^k_n)_{n=1}^\infty$ is not canceled out by $(x^l_n)_{n=1}^\infty$ for $l\neq k$ in the process of combining the rows of this array.

In the proof of Theorem \ref{fold} in \cite{KO}, there exist integers $k_1<k_2<\ldots$ such that the badness of row $n$ is witnessed by a linear combination of $k_n$ vectors. Then one can use the triangle inequality to uniformly control all linear combinations of at most $k_n$ vectors on each row.  Our proof of Theorem \ref{main} and \ref{ffold} will follow from a single result (Theorem \ref{good}) which interpolates therebetween, and is a transfinite version of the proof of Theorem \ref{fold}. This yields a unified approach to the two distinct Theorems \ref{KO}, \ref{fold} by quantifying using possibly infinite ordinals the complexity of linear combinations required to witness the non-uniformity of domination of subsequenes of members of $B_R$ by subsequences of $(g_n)_{n=1}^\infty$.  Theorem \ref{main} is the $\xi=\omega_1$ case and Theorem \ref{ffold} is the $\xi=\omega$ case of Theorem \ref{good}. In Section $4$, we discuss the distinction between the two extreme cases and the intermediate values of $\xi$.

Theorem \ref{KO} and Theorem \ref{Freeman} were shown by first proving the result for $C(K)$ spaces where $K$ is a countable, compact, Hausdorff space, and then proving the result in the general case by transporting the problem from a general $X$ into such a $C(K)$ space. Our approach avoids some of the technical difficulties incurred by transporting to an appropriate $C(K)$ space, and provides further, quantitative information not contained in those proofs. More precisely, if the conclusion of Theorem \ref{main} does not hold, one can produce a quantified measure of how the upper estimates tend to infinity on sets of prescribed complexity.

We note that neither of Theorems \ref{Freeman} and \ref{main} implies the other.  Our method of proof, however, can be adapted to provide an alternative proof of Theorem \ref{Freeman}. The final section of the paper compares Theorem \ref{main} to Theorem \ref{Freeman}, exhibiting many triples $X$, $R$, $(g_n)_{n=1}^\infty$ which satisfy the hypotheses of Theorem \ref{main}, but such that if $X$, $R$, $(v_n)_{n=1}^\infty$ satisfy the hypotheses of Freeman's theorem, then $(v_n)_{n=1}^\infty$ is equivalent to the canonical $\ell_1$ basis. Thus we show that our Theorem \ref{main} is genuinely distinct from Freeman's theorem. 

\section{Preliminaries}

Throughout, let $R$ be as described in the introduction. That is, $c_{00}(X)\subset R\subset \ell_\infty(X)$, $(R, \|\cdot\|_R)$ is a Banach space, $B_R\subset B_{\ell_\infty(X)}$,  and any subsequence of a member of $B_R$ is also a member of $B_R$.

Let us point out that for Theorems \ref{main} and \ref{ffold}, and for the later Theorem \ref{good}, it is sufficient to prove the theorem for a basis equivalent to $(g_n)_{n=1}^\infty$. Therefore we can and do assume that $(g_n)_{n=1}^\infty$ is normalized and bimonotone.  Also, since right dominance implies $r$-right dominance for some $1\leqslant r<\infty$, we will fix $1\leqslant r<\infty$ and assume throughout that $(g_n)_{n=1}^\infty$ is normalized, bimonotone, and $r$-right dominant.

In this work, we identify subsets of $\nn$ with strictly increasing sequences in the usual way.  We let $[\nn]^{<\omega}$ (resp. $[\nn]$) denote the finite (resp. infinite) subsets of $\nn$. For $M\in[\nn]$, we let $[M]^{<\omega}$ (resp. $[M]$) denote the set of finite (resp. infinite) subsets of $M$.   Given $E\subset \nn$ and $n\in\nn$, we write $n \leqslant E$ (resp. $n<E$) to mean that $n\leqslant \min E$ (resp. $n<\min E$).  We agree to the convention that $\min \varnothing=\infty$, so $n\leqslant \varnothing$ holds for any $n\in\nn$. 

Given $M\in[\nn]$, we let $M(n)$ denote the $n^{th}$ smallest member of $M$. If $F\in [\nn]^{<\omega}$, then for $n\leqslant |F|$, we let $F(n)$ denote the $n^{th}$ smallest member of $F$, so $F=(F(n))_{n=1}^{|F|}$.  Given $M\in[\nn]$ and $F\subset \nn$, we let $M(F)=(M(n): n\in F)$.  Given a collection $\mathcal{F}$ of subsets of $\nn$ and a subset $M$ of $\nn$, we let $\mathcal{F}\upp M$ denote the subset of $\mathcal{F}$ consisting of those members of $\mathcal{F}$ which are subsets of $M$.

Given a set $\Lambda$, we let $\Lambda^{<\omega}$ denote the set of finite sequences of members of $\Lambda$. We say $T\subset \Lambda^{<\omega}$ is a \emph{tree} provided that if $t\in T$ and $s$ is an initial segment of $t$, then $s\in T$. In particular, any non-empty tree includes the empty sequence, $\varnothing$.   Given a tree $T$, we  let $MAX(T)$ denote the set of maximal members of $T$ with respect to initial segment ordering. We then define the derived tree $T'$ of $T$ by $T'=T\setminus MAX(T)$, and note that $T'$ is also a tree.    We define by transfinite induction the derived trees $$T^0=T,$$ $$T^{\xi+1}=(T^\xi)',$$ and if $\xi$ is a limit ordinal, $$T^\xi=\bigcap_{\zeta<\xi}T^\zeta.$$   If there exists an ordinal $\xi$ such that $T^\xi=\varnothing$, then we say $T$ is \emph{well-founded} and define $\text{rank}(T)$ to be the minimum ordinal $\xi$ such that $T^\xi=\varnothing$. If for all $\xi$, $T^\xi\neq \varnothing$, then we say $T$ is \emph{ill-founded}. We have the following standard result regarding trees on countable sets $\Lambda$.  The statement and its proof are well-known, so we omit them.

\begin{theorem} If $\Lambda$ is a countable set and $T\subset \Lambda^{<\omega}$ is a tree, then either $T$ is ill-founded or $\text{\emph{rank}}(T)$ is countable. Furthermore, $T$ is ill-founded if and only if there exists a sequence $(\lambda_n)_{n=1}^\infty \subset \Lambda$ such that $(\lambda_n)_{n=1}^t\in T$ for all $t\in\nn$. 

\label{tree}
\end{theorem}

Given $(m_n)_{n=1}^t, (l_n)_{n=1}^t\in [\nn]^{<\omega}$, we say $(l_n)_{n=1}^t$ is a \emph{spread} of $(m_n)_{n=1}^t$ if $m_n\leqslant l_n$ for all $1\leqslant n\leqslant t$. 
 We say a subset $\mathcal{F}$ of $[\nn]^{<\omega}$ is \emph{spreading} if it contains all spreads of its members. We say a subset $\mathcal{F}$ of $[\nn]^{<\omega}$ is \emph{hereditary} if it contains all subsets of its members.  We collect the following standard facts about the Schreier and fine Schreier families, which can be found, for example, in \cite[Propositions $3.1$, $3.2$]{Con}. We note that our identification of subsets of $\nn$ with increasing sequences, we may view  a hereditary set $\mathcal{F}\subset [\nn]^{<\omega}$ as a tree on $\nn$ and consider its rank.  

We will need the following dichotomy, shown in \cite{G}. 

\begin{theorem} If $ \mathcal{F}, \mathcal{G}\subset [\nn]^{<\omega}$ are hereditary, then for any $L\in[\nn]$, there exists $M\in[L]$ such that either $\mathcal{F}\upp M\subset \mathcal{G}$ or $\mathcal{G}\upp M\subset \mathcal{F}$. 

\label{gasp}
\end{theorem}

 We recall the \emph{fine Schreier families} $(\mathcal{F}_\xi)_{\xi<\omega_1}$.  We let $$\mathcal{F}_0=\{\varnothing\},$$ $$\mathcal{F}_{\xi+1}= \{\varnothing\}\cup \{(n)\smallfrown F: n<F\in \mathcal{F}_\xi\},$$ and if $\xi<\omega_1$ is a limit ordinal, we fix $\xi_n\uparrow \xi$ and define $$\mathcal{F}_\xi=\{F: \exists n\leqslant F\in \mathcal{F}_{\xi_n}\}.$$

We also recall the \emph{Schreier families} $(\mathcal{S}_\xi)_{\xi<\omega_1}$.  We let $$\mathcal{S}_0=\mathcal{F}_1,$$ $$\mathcal{S}_{\xi+1}=\Bigl\{\bigcup_{n=1}^t F_n: t\leqslant F_1<\ldots <F_t, \varnothing\neq F_n\in \mathcal{S}_\xi\Bigr\},$$ and if $\xi$ is a limit ordinal, we fix $\xi_n\uparrow \xi$ and let $$\mathcal{S}_\xi=\{F: \exists n\leqslant F\in\mathcal{S}_{\xi_n}\}.$$

\begin{proposition}\begin{enumerate}[(i)]\item For every $\xi<\omega_1$, $\text{\emph{rank}}(\mathcal{F}_\xi)=\xi+1$ and $\text{\emph{rank}}(\mathcal{S}_\xi)=\omega^\xi+1$. \item $\mathcal{F}_\xi$ and $\mathcal{S}_\xi$ are hereditary and spreading.  \item The families $\mathcal{F}_\xi$, $\xi<\omega_1$,  have the \emph{almost monotone property}. That is, for each $\zeta<\xi\leqslant \omega_1$, there exists $l\in\nn$ such that $l<F\in \mathcal{F}_\zeta$ implies $F\in\mathcal{F}_\xi$.  \end{enumerate}

\label{sfacts}

\end{proposition}

For convenience, we let $\mathcal{F}_{\omega_1}= [\nn]^{<\omega}$. Of course, $\mathcal{F}_{\omega_1}$ is hereditary, spreading and contains each $\mathcal{F}_\xi$, $\xi<\omega_1$.

\section{Proof of Main Theorem}

Throughout this section, $1\leqslant r<\infty$ is fixed and  $(g_n)_{n=1}^\infty$ is a fixed, normalized, bimonotone, $r$-right dominant Schauder basis.  The Banach space  $X$ is fixed and $(R, \|\cdot\|_R)$ is a fixed subsequential space on $X$. We recall the definition of subsequential space:  \begin{enumerate}[(i)]\item $c_{00}(X)\subset R\subset \ell_\infty(X)$, \item $(R, \|\cdot\|_R)$ is a Banach space, \item any subsequence of a member of $B_R$ is a member of $B_R$, \item $B_R\subset B_{\ell_\infty(X)}$.\end{enumerate}

For $t\in\nn\cup \{\infty\}$ and sequences $(x_n)_{n=1}^t$, $(y_n)_{n=1}^t$ in (possibly different) Banach spaces, we write $(x_n)_{n=1}^t\leqslant_C (y_n)_{n=1}^t$ if $$\|\sum_{n=1}^t a_nx_n\|\leqslant C\|\sum_{n=1}^t a_ny_n\|$$ for all finitely supported scalar sequences $(a_n)_{n=1}^t$. Of course, if $(x_n)_{n=1}^t\leqslant_{C_1} (y_n)_{n=1}^t \leqslant_{C_2}(z_n)_{n=1}^t$, then $(x_n)_{n=1}^t\leqslant_{C_1C_2}(z_n)_{n=1}^t$.

Given a sequence $\varrho=(x_n)_{n=1}^\infty \in R$ and $0<C<\infty$, let $$T(\varrho, C)=\{\varnothing\}\cup \{(m_n, l_n)_{n=1}^t \in (\nn\times \nn)^{<\omega}: (x_{m_n})_{n=1}^t \leqslant_C (g_{l_n})_{n=1}^t, m_1<\ldots <m_t, \text{\ and\ }l_1<\ldots <l_t\}.$$

For $\varrho\in R$ and an ordinal $\xi\leqslant \omega_1$, we let $\Gamma(\xi, \varrho)$ denote the infimum of $C>0$ such that there exist $M,L\in[\nn]$ such that $$\{(M(F), L(F)): F\in\mathcal{F}_\xi\}\subset T(\varrho, C)$$ if such a $C$ exists, and  $\Gamma(\xi,\varrho)=\infty$ if no such $C$ exists.     We define $$\Gamma(\xi)=\sup_{\varrho\in B_R} \Gamma(\xi, \varrho).$$   We observe that for the $\xi=\omega_1$ case, since $\mathcal{F}_{\omega_1}=[\nn]^{<\omega}$,  $\Gamma(\omega_1, \varrho)$ is simply the infimum of $C>0$ such that $\varrho$ admits a subsequence $C$-dominated by a subsequence of $(g_n)_{n=1}^\infty$ if such a $C$ exists, and $\Gamma(\omega_1, \varrho)=\infty$ otherwise. Similarly, $\Gamma(\omega_1)$ is the infimum of constants $C$, if any such a $C$ exist, such that every member of $B_R$ has a subsequence $C$-dominated by a subsequence of $(g_n)_{n=1}^\infty$.  Therefore one can restate Theorem \ref{main} in the equivalent way: If $\Gamma(\omega_1, \varrho)<\infty$ for every $\varrho\in R$, then $\Gamma(\omega_1)<\infty$.

We now state our interpolation between Theorems \ref{main} and \ref{main2}, from which we will deduce both theorems as natural consequences. The statement of the theorem invites direct comparison with the Principle of Uniform Boundedness. 

\begin{theorem} Fix $\zeta\leqslant \omega_1$. If $\Gamma(\zeta, \varrho)<\infty$ for each $\varrho\in R$, then $\Gamma(\zeta)<\infty$. 

\label{good}
\end{theorem}

\begin{rem}\upshape Note that if $\Gamma(\xi)<C$, then for any $K\in[\nn]$ and $\varrho\in B_R$, there exist $M\in[K]$ and $L\in[\nn]$ such that $$\{(M(F), L(F)): F\in \mathcal{F}_\xi\}\subset T(\varrho, C).$$ That is, the definition of $\Gamma(\xi)$ yields that there exists such a set $M$ which is a subset of $\nn$, but our hypotheses yield that, once $K\in[\nn]$ is fixed, $M$ can be taken to be a subset of $K$. To see this, note that if $\varrho=(x_n)_{n=1}^\infty$, then the subsequence $\varsigma=(y_n)_{n=1}^\infty = (x_{K(n)})_{n=1}^\infty\in B_R$ by the properties of $R$. From this and the definition of $\Gamma(\xi)$, it follows that there exist $N,L\in[\nn]$ such that $$\{(N(F), L(F)): F\in\mathcal{F}_\xi\}\subset T(\varsigma, C).$$  Now if $M(n)=K(N(n))$ for all $n\in\nn$, then $M\in[K]$ and for each $F\in\mathcal{F}_\xi$, $$(x_{M(n)})_{n\in F}= (x_{K(N(n))})_{n\in F}=(y_{N(n)})_{n\in F}\leqslant_C (g_{L(n)})_{n\in F},$$ since $(N(F), L(F))\in T(\varsigma, C)$.  Therefore $$\{(M(F), L(F)):F\in\mathcal{F}_\xi\}\subset T(\varrho, C).$$  

\label{Am}
\end{rem}

\begin{lemma} \begin{enumerate}[(i)]\item For any $\varrho\in R$ and $\zeta\leqslant \xi\leqslant \omega_1$, $\Gamma(\zeta, \varrho)\leqslant \Gamma(\xi, \varrho)$. Consequently, $\Gamma(\zeta)\leqslant \Gamma(\xi)$. \item For any limit ordinal $\xi\leqslant \omega_1$,  $\Gamma(\xi)\leqslant r\sup_{\zeta<\xi}\Gamma(\zeta)$. \item   $\Gamma(0)=0$. \item For any $\zeta, \xi<\omega_1$, $\Gamma(\zeta+\xi)\leqslant r(\Gamma(\zeta)+\Gamma(\xi))$. \end{enumerate} 

\label{work}

\end{lemma}

\begin{proof} We will show several inequalities in the proof. If the majorizing quantity is infinite in any of these inequalities, the inequality holds trivially. Therefore we omit this trivial case in each inequality in the proof.

$(i)$  We first show that for a fixed $\varrho\in R$, $\xi\mapsto \Gamma(\xi, \varrho)$ is non-decreasing. Fix $\zeta\leqslant \xi\leqslant \omega_1$. Suppose that $\Gamma(\xi, \varrho)<C<\infty$.   Fix $M_1, L_1\in[\nn]$ such that $$\{(M_1(F), L_1(F)): F\in\mathcal{F}_\xi\}\subset T(\varrho, C).$$   By the almost monotone property of the fine Schreier families (or from the fact that $\mathcal{F}_\zeta\subset \mathcal{F}_{\omega_1}$ if $\xi=\omega_1$), there exists $l\in\nn$ such that if $l<F\in\mathcal{F}_\zeta$, then  $F\in\mathcal{F}_\xi$.   Let $M(n)=M_1(n+l)$ and $L(n)=L_1(n+l)$ for all $n\in\nn$.   Now fix $F\in\mathcal{F}_\zeta$ and let $G=(n+l: n\in F)$. Since $G$ is a spread of $F$, $G\in\mathcal{F}_\zeta$. Since $l<G$, $G\in\mathcal{F}_\xi$.  Then $$(x_{M(n)})_{n\in F} = (x_{M_1(n+l)})_{n\in F}= (x_{M_1(n)})_{n\in G}\leqslant_C (g_{L_1(n)})_{n\in G} = (g_{L_1(n+l)})_{n\in F}= (g_{L(n)})_{n\in F}.$$  Thus $$\{(M(F), L(F)):F\in\mathcal{F}_\zeta\}\subset T(\varrho, C),$$ and $\Gamma(\zeta, \varrho)\leqslant C$. Since $C>\Gamma(\xi, \varrho)$ was arbitrary, $\Gamma(\zeta, \varrho)\leqslant \Gamma(\xi, \varrho)$.  For the second statement of $(i)$, we note that $$\Gamma(\zeta)=\sup_{\varrho\in B_R} \Gamma(\zeta, \varrho) \leqslant \sup_{\varrho\in B_R} \Gamma(\xi, \varrho)=\Gamma(\xi).$$

$(ii)$ We next show that for each limit ordinal $\xi\leqslant \omega_1$, $\Gamma(\xi)\leqslant r\sup_{\zeta<\xi}\Gamma(\zeta)$.  Assume that $\xi\leqslant \omega_1$ is a limit ordinal and $\sup_{\zeta<\xi}\Gamma(\zeta)<C<\infty$.  We consider the cases $\xi<\omega_1$ and $\xi=\omega_1$.   

First suppose that $\xi<\omega_1$. Let $\xi_k\uparrow \xi$ be such that $$\mathcal{F}_\xi=\{F: \exists k\leqslant F\in \mathcal{F}_{\xi_k}\}.$$ Fix $\varrho\in B_R$. As noted in  Remark \ref{Am}, we may select $M_1\supset M_2\supset \ldots$ and $L_1, L_2, \ldots$ such that for all $k\in\nn$, $$\{(M_k(F), L_k(F)): F\in\mathcal{F}_{\xi_k}\}\subset T(\varrho, C).$$  That is, the content of Remark \ref{Am} explains how  we know that $M_{n+1}$ may be chosen as a subset of $M_n$.  For each $n\in\nn$, we note that $M_n(n)\in \cap_{k=1}^n M_k$. Therefore  we can choose integers $s^n_1, \ldots, s^n_n$ such that $M_n(n)=M_k(s^n_k)$ for each $1\leqslant k\leqslant n$. For $1\leqslant k\leqslant n\in\nn$,  since the $n^{th}$ smallest member of $M_n$ cannot occur before the $n^{th}$ position in $M_k$, $s^n_k\geqslant n$ for each $1\leqslant k\leqslant n$.  For each $n\in\nn$, let $L(n)=\max\{L_1(s^n_1), L_2(s^n_2), \ldots, L_n(s^n_n)\}$ and let $M(n)=M_n(n)$. We claim that $$\{(M(F), L(F)): F\in\mathcal{F}_\xi\}\subset T(\varrho, rC).$$  To see this, fix $F\in\mathcal{F}_\xi$ and fix $k\in\nn$ such that  $k\leqslant F\in \mathcal{F}_{\xi_k}$.  Since $s^n_k$ is defined for each $ k\leqslant n\in \nn$ and since $k\leqslant F$, $s^n_k$ is defined for each $n\in F$.    Since $\mathcal{F}_{\xi_k}$ is spreading and since $s^n_k\geqslant n$ for each $ k\leqslant n\in\nn$, $G:=(s^n_k: n\in F)$ is a spread of $F$, and therefore also a member of $\mathcal{F}_{\xi_k}$. Then by $r$-right dominance and the properties of $M_k$, $L_k$, and since $L_k(s^n_k) \leqslant L(n)$ for each $n\in F$, $$(x_{M(n)})_{n\in F}= (x_{M_n(n)})_{n\in F} = (x_{M_k(s^n_k)})_{n\in F} =(x_{M_k(n)})_{n\in G} \leqslant_C (g_{L_k(n)})_{n\in G} = (g_{L_k(s^n_k)})_{n\in F} \leqslant_r (g_{L(n)})_{n\in F}.$$   This gives the desired inclusion and completes the $\xi<\omega_1$ case. 

Now suppose that $\xi=\omega_1$.   Fix $\varrho\in B_R$.   For  any $\zeta<\omega_1$, there exist $M_\zeta, L_\zeta\in [\nn]$ such that $$\{(M_\zeta(F), L_\zeta(F)): F\in \mathcal{F}_\zeta\}\subset T(\varrho, C).$$   The map $\mathcal{F}_\zeta\ni F\mapsto (M_\zeta(F), L_\zeta(F))\in T(\varrho, C)$ defines a tree embedding of $\mathcal{F}_\zeta$ into $T(\varrho, C)$. From this it follows that $\text{rank}(T(\varrho, C))\geqslant \text{rank}(\mathcal{F}_\zeta)=\zeta+1$. Since this is true for any $\zeta<\omega_1$, $T(\varrho, C)$ is ill-founded by Theorem \ref{tree}. From this it follows that there exist $M,L\in [\nn]$ such that $(M(n), L(n))_{n=1}^t\in T(\varrho, C)$ for all $t\in\nn$. Therefore $(x_{M(n)})_{n=1}^\infty\leqslant_C (g_{L(n)})_{n=1}^\infty$, and $$\{(M(F), N(F)): F\in \mathcal{F}_{\omega_1}\}\subset T(\varrho, C).$$   Thus $\Gamma(\omega_1)\leqslant C\leqslant rC$, and this completes the $\xi=\omega_1$ case.

$(iii)$ Since $$\{\varnothing\}= \{(F,F): F\in\mathcal{F}_0\}\subset T(\varrho, C)$$ for any $C>0$, $\Gamma(0)=0$.

$(iv)$ Fix $\zeta, \xi<\omega_1$.  Fix $C_1>\Gamma(\zeta)$ and $C_2>\Gamma(\xi)$.  Fix $\varrho\in B_R$ and note that there exist $M_1, L_1\in [\nn]$ such that $$\{(M_1(F), L_1(F)):F\in\mathcal{F}_\zeta\}\subset T(\varrho, C_1).$$   As noted in Remark \ref{Am}, we can find $M_2\in [M_1]$ and $L_2$ such that $$\{(M_2(F), L_2(F)):F\in\mathcal{F}_\xi\}\subset T(\varrho, C_2).$$  For each $n\in\nn$, we can choose $s_n\in\nn$ such that $M_2(n)=M_1(s_n)$. It follows that $s_1<s_2<\ldots$ and for each $n\in\nn$,  $n\leqslant s_n$. For each $n\in\nn$, define $$L_3(n)=\max\{L_2(n), L_1(s_n)\}.$$   Note that for any $F\in \mathcal{F}_\zeta$ and $G\in \mathcal{F}_\xi$, with $F_2=(s_n: n\in F)\in \mathcal{F}_\zeta$,  $$(x_{M_2(n)})_{n\in F}= (x_{M_1(s_n)})_{n\in F}=(x_{M_1(n)})_{n\in F_2}\leqslant_{C_1} (g_{L_1(n)})_{n\in F_2}= (g_{L_1(s_n)})_{n\in F}\leqslant_r (g_{L_3(n)})_{n \in F}$$   and $$(x_{M_2(n)})_{n\in G} \leqslant_{C_2} (g_{L_2(n)})_{n\in G}\leqslant_r (g_{L_3(n)})_{n\in G}.$$

Let $$\mathcal{F}=\{G\cup F: G\in \mathcal{F}_\xi, F\in \mathcal{F}_\zeta,  G< F\}.$$ As shown in \cite[Proposition $3.1$]{Con}, $\mathcal{F}$ is regular with $$\text{rank}(\mathcal{F})=\zeta+\xi+1=\text{rank}(\mathcal{F}_{\zeta+\xi}),$$ which implies the existence of some $P\in[\nn]$ such that $$\{P(F):F\in \mathcal{F}_{\zeta+\xi}\}\subset \mathcal{F}.$$  For each $n\in\nn$,  define $M(n)=M_2(P(n))$   and $L(n)=L_3(P(n))$.  Fix $H\in \mathcal{F}_{\zeta+\xi}$ and scalars $(a_n)_{n\in H}$.  It follows from our choice of $P$ that $P(H)\in \mathcal{F}$, so that  $P(H)=G\cup F$ for some $G<F$ such that $G\in\mathcal{F}_\xi$ and $F\in \mathcal{F}_\zeta$.  Define $(b_n)_{n\in P(H)}=(b_{P(n)})_{n\in H}$ by letting $b_{P(n)}=a_n$.       Then by the preceding paragraph and the bimonotonicity of $(g_n)_{n=1}^\infty$, \begin{align*} \|\sum_{n\in H} a_n x_{M(n)}\| & = \|\sum_{n\in P(H)} b_n x_{M_2(n)}\| \leqslant \|\sum_{n\in G} b_n x_{M_2(n)} \|+\|\sum_{n\in F}b_nx_{M_2(n)}\|\\ &  \leqslant rC_2\|\sum_{n\in G}b_n g_{L_3(n)}\|+rC_1\|\sum_{n\in F} b_n g_{L_3(n)}\| \\ & \leqslant r(C_1+C_2) \|\sum_{n\in P(H)} b_n g_{L_3(n)}\| = r(C_1+C_2)\|\sum_{n\in H} a_n g_{L(n)}\|.\end{align*}   This shows that $$\{(M(F), L(F)): F\in\mathcal{F}_{\zeta+\xi}\}\subset T(\varrho, r(C_1+C_2)).$$ Since $C_1>\Gamma(\zeta)$ and $C_2>\Gamma(\xi)$ were arbitrary, we are done.

\end{proof}

\begin{corollary} Either $\{\xi\leqslant \omega_1: \Gamma(\xi)=\infty\}$ is empty or there exists $0<\gamma<\omega_1$ such that $\min \{\xi\leqslant \omega_1: \Gamma(\xi)=\infty\}=\omega^\gamma$.  Moreover, in the case that $\{\xi\leqslant \omega_1: \Gamma(\xi)=\infty\}\neq \varnothing$ and $\omega^\gamma=\min\{\xi\leqslant \omega_1: \Gamma(\xi)=\infty\}$, $\sup \{\Gamma(\xi): \xi<\omega^\gamma\}=\infty$.

\label{break}
\end{corollary}

\begin{proof} Assume $\varnothing\neq \{\xi\leqslant \omega_1: \Gamma(\xi)=\infty\}$.    Let $\mu=\min\{\xi\leqslant \omega_1: \Gamma(\xi)=\infty\}$. Note that by Lemma \ref{work}$(i)$, $\{\xi\leqslant \omega_1: \Gamma(\xi)=\infty\}=[\mu, \omega_1]$.  Since $\Gamma(0)=0$, $\mu>0$.    If $\zeta, \xi<\mu$, then by Lemma \ref{work}$(iv)$, $\Gamma(\zeta+\xi)\leqslant r(\Gamma(\zeta)+\Gamma(\xi))<\infty$, and $\zeta+\xi<\mu$.  From this and standard properties of ordinals, there exists $\gamma$ such that $\mu=\omega^\gamma$. To complete the first statement, it remains to show that $0<\gamma<\omega_1$. Since $\mathcal{F}_1$ consists of sets of cardinality at most $1$, $\Gamma(1)\leqslant 1$, from which it follows that $\Gamma(1)=\Gamma(\omega^0)<\infty$, and $\gamma>0$.   If $\sup_{\xi<\omega_1}\Gamma(\xi)<\infty$, then by Lemma \ref{work}$(i)$ and $(ii)$,  $$\infty =\Gamma(\mu) \leqslant \Gamma(\omega_1)=\Gamma(\omega_1)\leqslant r\sup_{\xi<\omega_1} \Gamma(\xi)<\infty,$$ a contradiction.  Therefore $\sup_{\xi<\omega_1}\Gamma(\xi)=\infty$.  From this it follows that for each $n\in\nn$, $\{\xi<\omega_1: \Gamma(\xi)>n\}$ is non-empty, and $$\nu:=\sup_n \min \{\xi<\omega_1: \Gamma(\xi)>n\}<\omega_1.$$  Again by Lemma \ref{work}$(i)$, $\Gamma(\nu)=\infty$, so $\omega^\gamma=\mu\leqslant \nu<\omega_1$.  From this it follows that $\gamma<\omega_1$. 

From Lemma \ref{work}$(ii)$, $$\infty= \frac{1}{r} \infty= \frac{1}{r}\Gamma(\omega^\gamma)\leqslant \sup_{\xi<\omega^\gamma} \Gamma(\xi).$$

\end{proof}

The next corollary is what we will use to choose the rows of our array which we combine to produce a contradiction in our proof of Theorem \ref{good}.  We note that the hypotheses $(a)$-$(c)$ of the following corollary are the contradiction hypotheses for our eventual proof of Theorem \ref{good}, and, as we shall see, it is not possible for these three hypotheses to simultaneously hold.

\begin{corollary} Suppose that for $0<\mu<\omega_1$, \begin{enumerate}[(a)]\item $\Gamma(\xi)<\infty$ for all $\xi<\mu$, \item $\Gamma(\mu)=\infty$, \item for each $\varrho\in R$, $\Gamma(\mu, \varrho)<\infty$. \end{enumerate} Then for any $D>0$ and any finite subset $S$ of $[0, \mu)$, there exist $\xi<\mu$, $\varrho=(x_n)_{n=1}^\infty\in B_R$, $N\in[\nn]$, and $C<\infty$ such that \begin{enumerate}[(i)]\item for any $F\in \mathcal{F}_\mu$ and scalars $(a_n)_{n\in F}$, $\|\sum_{n\in F} a_nx_n\|\leqslant C\|\sum_{n\in F}a_ng_{N(n)}\|$, \item for any $\zeta\in S$, $F\in \mathcal{F}_\zeta$, and scalars $(a_n)_{n\in F}$, $\|\sum_{n\in F}a_nx_n\|\leqslant (r\Gamma(\zeta)+\frac{1}{r})\|\sum_{n\in F}a_ng_{N(n)}\|$, \item for any $M,L\in[\nn]$, there exist $F\in \mathcal{F}_\xi\cap \mathcal{F}_\mu$ and scalars $(a_n)_{n\in F}$ such that $\|\sum_{n\in F}a_n x_{M(n)}\|>D\|\sum_{n\in F}x_n g_{L(n)}\|$. \end{enumerate}

\label{home}
\end{corollary}

\begin{proof}  By Lemma \ref{work}$(ii)$, condition $(b)$ yields that there exists $\xi<\mu$ such that $\Gamma(\xi)>D$.    By the definition of $\Gamma(\xi)$, there exists $\varsigma=(y_n)_{n=1}^\infty \in B_R$ such that for any $P,Q\in[\nn]$, $$\{(P(F), Q(F)): F\in\mathcal{F}_\xi\}\not\subset T(\varsigma, D).$$

Condition $(c)$ yields the existence of some $M_0, L_0\in [\nn]$ and $C<\infty$ such that $$\{(M_0(F), L_0(F)): F\in\mathcal{F}_\mu\}\subset T(\varsigma, C/r).$$   If $S=\varnothing$,  let $K=M_0$ and $N=L_0$.  Otherwise we enumerate $S=\{\xi_1, \ldots, \xi_t\}$ and recursively select $M_0\supset M_1\supset \ldots \supset M_t\in[\nn]$ and $L_1, \ldots, L_t$ such that for each $1\leqslant i\leqslant t$, $$\{(M_i(F), L_i(F)): F\in\mathcal{F}_{\xi_i}\}\subset T(\varsigma, \Gamma(\xi_i)+1/r^2).$$  We recall that we are able to choose $M_{n+1}\in [M_n]$, as noted in Remark \ref{Am}.   For each $0\leqslant i\leqslant t$ and $n\in\nn$, choose $s_i^n\in\nn$ such that $M_t(n)=M_i(s_i^n)$ and note that $s_i^n\geqslant n$.    For each $n\in\nn$, let $K(n)=M_t(n)$ and $$N(n)=\max\{L_0(s^n_0), L_1(s_1^n),  \ldots, L_t(s_t^n)\}.$$    

Let $\varrho=(x_n)_{n=1}^\infty=(y_{K(n)})_{n=1}^\infty\in B_R$. We now verify that $(i)$, $(ii)$, and $(iii)$ are satisfied with this $\varrho$, $\xi$, $N$, $C$.  For $F\in\mathcal{F}_\mu$, $G:= (s_0^n)_{n\in F}\in \mathcal{F}_\mu$. Note that $K(F)=M_0(G)$, and since $(M_0(G), L_0(G))\in T(\varsigma, C/r)$ and $N(n)\geqslant L_0(s_0^n)$ for all $n\in\nn$, condition $(i)$ is satisfied.   Indeed, fix scalars $(a_n)_{n\in F}$, let $b_{s_0^n}=a_n$ for $n\in F$, and observe that  \begin{align*} \|\sum_{n\in F}a_n x_n\| & = \|\sum_{n\in F} a_n y_{K(n)}\| = \|\sum_{n\in G} b_n y_{M_0(n)}\| \leqslant \frac{C}{r}\|\sum_{n\in G} b_n g_{L_0(n)}\| = \frac{C}{r} \|\sum_{n\in F}a_ng_{L_0(s_0^n)}\| \\ &  \leqslant C\|\sum_{n\in F}a_n g_{N(n)}\|.\end{align*}

If $S=\varnothing$, $(ii)$ is vacuous. If $S\neq \varnothing$, fix $1\leqslant i\leqslant t$, $F\in\mathcal{F}_{\xi_i}$, and let $G=(s_i^n)_{n\in F}\in \mathcal{F}_{\xi_i}$. Since $K(F)=M_i(G)$, and since $(M_i(G), L_i(G))\in T(\varsigma, \Gamma(\xi_i)+1/r^2)$ and $N(n)\geqslant L_i(s_i^n)$ for all $n\in\nn$, condition $(ii)$ is satisfied, analogously to the last part of the previous paragraph.

We now prove $(iii)$. Now fix $M,L\in[\nn]$.  By the almost monotone property of the fine Schreier families, there exists $l\in \nn$ such that $l<F\in \mathcal{F}_\xi$ implies $F\in \mathcal{F}_\mu$.    For each $n\in\nn$, let $P_1(n)=M(n+l)$, $P(n)=K(P_1(n))$, and  $Q(n)=L(n+l)$.  By our choice of $\varsigma$,  $$\{(P(F), Q(F)): F\in\mathcal{F}_\xi\}\not \subset T(\varsigma, D).$$ From this it follows that there exist $G\in \mathcal{F}_\xi$ and scalars $(b_n)_{n\in G}$ such that $\|\sum_{n\in G} b_n y_{P(n)} \| >D\|\sum_{n\in G} b_n g_{Q(n)}\|$. Let $F=(n+l: n\in G)\in \mathcal{F}_\xi\cap \mathcal{F}_\mu$.  Define $a_{n+l}=b_n$ for each $n\in G$. Then \begin{align*}  \|\sum_{n\in F} a_n x_{M(n)}\| & = \|\sum_{n\in G} a_{n+l} x_{M(n+l)}\| = \|\sum_{n\in G}b_n y_{K(P_1(n))}\| = \|\sum_{n\in G} b_ny_{P(n)}\| \\ & >D\|\sum_{n\in G}b_n g_{Q(n)}\|=D\|\sum_{n\in G} a_{n+l} g_{L(n+l)}\| = D\|\sum_{n\in F}a_n g_{L(n)}\|. \end{align*}

\end{proof}

\begin{proof}[Proof of Theorem \ref{good}] Seeking a contradiction, assume $\zeta\leqslant \omega_1$ is such that for every $\varrho\in R$, $\Gamma(\zeta, \varrho)<\infty$, but $\Gamma(\zeta)=\infty$.  Let $\mu=\min\{\xi\leqslant \omega_1: \Gamma(\xi)=\infty\}\leqslant \zeta$. By Lemma \ref{work}$(i)$, $\Gamma(\mu, \varrho)<\infty$ for each $\varrho\in R$.   Moreover, by Lemma \ref{work}, items $(a)$, $(b)$, and $(c)$ of Corollary \ref{home} are satisfied.  

Fix $D_1\geqslant 4^1$ and apply Corollary \ref{home} with $D=D_1$ and $S=\varnothing$ to find $\varrho_1\in B_R$, $\xi_1<\mu$, $N_1\in[\nn]$, and $C_1<\infty$ satisfying $(i)$-$(iii)$ of Corollary \ref{home}.

Assume constants $D_1, \ldots, D_{k-1}$, $C_1, \ldots, C_{k-1}$, sequences $\varrho_1, \ldots, \varrho_{k-1}\in B_R$, sets $N_1, \ldots, N_{k-1}\in[\nn]$, and ordinals $\xi_1, \ldots, \xi_{k-1}<\mu$ have been chosen.  Fix $D_k\geqslant 4^k$ such that \begin{equation}\sum_{l=1}^{k-1} l+rC_l<D_k^{1/2}/3\tag{1}\end{equation} and \begin{equation}\underset{1\leqslant l<l}{\max} \frac{r^2\Gamma(\xi_l)+l}{D_k^{1/2}D_l^{1/2}}< \frac{1}{3\cdot 2^k}.\tag{2}\end{equation}    Now apply Corollary \ref{home} with $D=D_k$ and $S=\{\xi_1, \ldots, \xi_{k-1}\}$ to obtain a sequence $\varrho_k\in B_R$, a constant $C_k$, an ordinal $\xi_k<\mu$, and a set $N_k\in[\nn]$ satisfying the conclusion of Corollary \ref{home}.   This completes the recursive construction.

For each $l\in\nn$, let $\varrho_l=(x^l_n)_{n=1}^\infty$. Let $$\varrho=\sum_{l=1}^\infty \frac{1}{D_l^{1/2}}\varrho_l\in B_R$$ and for each $n\in\nn$, $$x_n=\sum_{l=1}^\infty \frac{1}{D_l^{1/2}} x^l_n.$$ Note that $\varrho=(x_n)_{n=1}^\infty$.   Here we are using that $(R, \|\cdot\|_R)$ is a Banach space and $$\sum_{l=1}^\infty \frac{1}{D_l^{1/2}}\leqslant \sum_{l=1}^\infty \frac{1}{2^l} =1.$$  

By our hypotheses, $\Gamma(\mu, \varrho)<\infty$.     This means there exist $0<C<\infty$ and $M,L_0\in[\nn]$ such that for any $F\in\mathcal{F}_\mu$ and scalars $(a_n)_{n\in F}$, $\|\sum_{n\in F}a_n x_{M(n)}\|\leqslant (C/r)\|\sum_{n\in F} a_n g_{L_0(n)}\|$.  For each $n\in\nn$, let $$L(n)=\max\{L_0(n), N_1(M(n)), \ldots, N_n(M(n))\}$$   and note that for any $F\in\mathcal{F}_\mu$ and scalars $(a_n)_{n\in F}$, $\|\sum_{n\in F}a_n x_{M(n)}\|\leqslant C\|\sum_{n\in F} a_n g_{L(n)}\|$.  Now fix $k\in\nn$ so large that $D_k^{1/2}/3>C$.     By our recursive choices, namely the fact that our choices satisfy Corollary \ref{home}$(iii)$, there exist $F\in \mathcal{F}_{\xi_k}\cap \mathcal{F}_\mu$ and scalars $(a_n)_{n\in F}$ such that $\|\sum_{n\in F}a_n g_{L(n)}\|=1$ and $\|\sum_{n\in F}a_n x^k_{M(n)}\|>D_k$.

For $l\in\nn$,  if $|F|<l$, let $E_l=F$ and $F_l=\varnothing$. Otherwise let $E_l$ denote the set consisting of the $l-1$ smallest members of $F$ and let $F_l=F\setminus E_l$.   It follows from this definition that for each $l\in\nn$, $l\leqslant F_l$.   Note that $\max_{n\in F}|a_n|\leqslant 1$ and $\|\sum_{n\in F_l} a_ng_{L(n)}\|\leqslant 1$ since $(g_n)_{n=1}^\infty$ is normalized and bimonotone.  For any $l\in\nn$, since $\varrho_l\in B_R$ and $\max_{n\in F}|a_n|\leqslant 1$, $$\|\sum_{n\in E_l}a_n x^l_{M(n)}\|\leqslant |E_l|\max_{n\in E_l} |a_n| \leqslant l-1.$$

Fix $l<k$.  Note that  $F_l\in \mathcal{F}_\mu$, from which it follows that $M(F_l)\in \mathcal{F}_\mu$. Furthermore, for any $n\in F_l$, since $l\leqslant n$, $$L(n)=\max\{L_0(n), N_1(M(n)), \ldots, N_n(M(n))\}\geqslant N_l(M(n)).$$   Then by the properties of $\varrho_l$, $$(x^l_{M(n)})_{n\in F_l} = (x^l_n)_{n\in M(F_l)} \leqslant_{C_l} (g_{N_l(n)})_{n\in M(F_l)}= (g_{N_l(M(n))})_{n\in F_l}\leqslant_r (g_{L(n)})_{n\in F_l}.$$   Therefore $$\|\sum_{n\in F}a_n x^l_{M(n)}\| \leqslant l-1+\|\sum_{n\in F_l}a_n x^l_{M(n)}\|\leqslant l+rC_l \|\sum_{n\in F_l} a_ng_{L(n)}\|\leqslant l+rC_l.$$   

Now fix $l>k$. Note that  $F_l\in \mathcal{F}_{\xi_k}$, from which it follows that $M(F_l)\in \mathcal{F}_{\xi_k}$. Furthermore, for any $n\in F_l$, since $l\leqslant n$, $$L(n)=\max\{L_0(n), N_1(M(n)), \ldots, N_n(M(n))\}\geqslant N_l(M(n)).$$   Then by the properties of $\varrho_l$, $$(x^l_{M(n)})_{n\in F_l} = (x^l_n)_{n\in M(F_l)}\leqslant_{r\Gamma(\xi_k)+1/r} (g_{N_l(n)})_{n\in M(F_l)}=(g_{N_l(M(n))})_{n\in F_l} \leqslant_r (g_{L(n)})_{n\in F_l}.$$   Therefore using $(2)$, $$\|\sum_{n\in F}a_n x^l_{M(n)}\| \leqslant l-1+\|\sum_{n\in F_l}a_Nx^l_{M(n)}\|\leqslant l-1+r(r\Gamma(\xi_k)+1/r) \|\sum_{n\in F_l} a_ng_{L(n)}\|\leqslant r^2\Gamma(\xi_k)+l.$$

Then, using $(1)$ and $(2)$,  \begin{align*} C & =C\|\sum_{n\in F}a_ng_{L(n)}\| \geqslant \|\sum_{n\in F}a_nx_{M(n)}\|\\ &  \geqslant \frac{1}{D_k^{1/2}}\|\sum_{n\in F}a_nx^k_{M(n)}\| - \sum_{l=1}^{k-1} \frac{1}{D_l^{1/2}}\|\sum_{n\in F}a_n x^l_{M(n)}\|   - \sum_{l=k+1}^\infty \frac{1}{D_l^{1/2}} \|\sum_{n\in F}a_n x^l_{M(n)}\|  \\ & \geqslant D_k^{1/2}-\sum_{l=1}^{k-1} l+rC_l - \sum_{l=k+1}^\infty \frac{r^2\Gamma(\xi_k)+l}{D_l^{1/2}} \\ & \geqslant D_k^{1/2} - \sum_{l=1}^{k-1} l+rC_l - \sum_{l=k+1}^\infty \frac{D_k^{1/2}}{3\cdot 2^l} \\ & \geqslant D_k^{1/2}-D_k^{1/2}/3-D_k^{1/2}/3 = D_k^{1/2}/3>C.   \end{align*} This contradiction finishes the proof.

\end{proof}

\begin{proof}[Proof of Theorem \ref{main}]

Since $\mathcal{F}_{\omega_1}=[\nn]^{<\omega}$, the hypothesis that every member of $R$ admits a subsequence dominated by a subsequence of $(g_n)_{n=1}^\infty$ is equivalent to the hypothesis that $\Gamma(\omega_1, \varrho)<\infty$ for all $\varrho\in R$. By Theorem \ref{good}, $\Gamma(\omega_1)<\infty$. This yields that for any $\Gamma(\omega_1)<C$, any member of $B_R$ admits a subsequence $C$-dominated by a subsequence of $(g_n)_{n=1}^\infty$.

\end{proof}

\section{Spreading models}

Let $(e_n)_{n=1}^\infty$ be a sequence in some seminormed vector space. Let us say that  that the sequence $(x_n)_{n=1}^\infty$ in some Banach space $X$ generates the $*$-\emph{spreading model} $(e_n)_{n=1}^\infty$ provided that for any $m\in\nn$ and scalars $(a_n)_{n=1}^m$, $$\underset{l_1\to \infty}{\lim} \ldots \underset{l_m\to \infty}{\lim} \|\sum_{n=1}^m a_n x_{l_n}\|= \|\sum_{n=1}^m a_n e_n\|.$$  We say a $*$-spreading model is a \emph{spreading model} if the sequence $(e_n)_{n=1}^\infty$ is normalized, basic, and contained in a normed, rather than just semi-normed, space.

 Standard results from Ramsey theory yield that any $(x_n)_{n=1}^\infty\in \ell_\infty(X)$ has a subsequence which generates some $*$-spreading model. Moreover, any normalized, bimonotone basic sequence has a subsequence which generates a spreading model, which is also evidently a normalized, bimonotone basic sequence.  It is also well-known that if $(e_n)_{n=1}^\infty$ is a $*$-spreading model generated by a  weakly null sequence $(x_n)_{n=1}^\infty$, and if $\lim_n \|x_n\|=1$, then $(e_n)_{n=1}^\infty$ is a spreading model. 

\begin{rem}\upshape In all of our examples, the space $R$ will be a subspace of $c_0^w(X)$, the space of weakly null sequences in $X$, endowed with the norm $\|\cdot\|_R=\|\cdot\|_{\ell_\infty(X)}$.    In this case, we  define $R_0=\{(x_n)_{n=1}^\infty\in B_R: (\forall n\in\nn)(\|x_n\|=1)\}$ and note that any $*$-spreading model generated by a member of $R_0$ is a spreading model. 

In the trivial case in which $R_0=\varnothing$, it must be the case that $R\subset c_0(X)$. In this case, every $*$-spreading model $(e_n)_{n=1}^\infty$ generated by a member of $R$ satisfies $\|\sum_{n=1}^m a_ne_n\|=0$ for all $m\in\nn$ and scalars $(a_n)_{n=1}^m$. Also, $\Gamma(\xi, \varrho)=0$ for any $\xi\leqslant \omega_1$ and $\varrho\in R$.   Thus all of the conditions in Theorem \ref{main2} are trivially satisfied in this case. 

In the case $R_0\neq \varnothing$, by homogeneity and standard perturbation arguments, the conditions that $\Gamma(\omega, \varrho)\leqslant C$ for every $\varrho\in B_R$ and $\Gamma(\omega, \varrho)\leqslant C$ for every $\varrho\in R_0$ are equivalent.  Similarly, for a fixed $C>0$, the hypothesis that for every $\ee>0$, every $*$-spreading model generated by a member of $B_R$ is $C+\ee$-dominated by a spreading model generated by a subsequence of $(g_n)_{n=1}^\infty$ is equivalent to the hypothesis that for every $\ee>0$, every spreading model generated by a member of $R_0$ is $C+\ee$-dominated by a spreading model generated by a subsequence of $(g_n)_{n=1}^\infty$. Therefore in all examples of interest, our hypotheses stated in terms of $B_R$ and $*$-spreading models are equivalent to hypothesis on members $(x_n)_{n=1}^\infty$ of $B_R$ with $\|x_n\|=1$ for all $n\in\nn$ and spreading models.

\end{rem}

If $(g_n)_{n=1}^\infty$ is a normalized, bimonotone, $r$-right dominant basis, then all spreading models generated by $(g_n)_{n=1}^\infty$ are uniformly equivalent. Indeed, if $(e_n)_{n=1}^\infty$ and $(f_n)_{n=1}^\infty$ are spreading models generated by $(g_{l_n})_{n=1}^\infty$ and $(g_{k_n})_{n=1}^\infty$, respectively, we can fix scalars $(a_n)_{n=1}^m$ and $p_0<\ldots <p_m$ and $q_1<\ldots <q_m$ such that $l_{p_{n-1}}<k_{q_n}<l_{p_n}$ for each $1\leqslant n\leqslant m$, and such that $$\|\sum_{n=1}^m a_n g_{l_{p_{n-1}}}\|, \|\sum_{n=1}^m a_n g_{l_{p_n}}\|\approx \|\sum_{n=1}^m a_ne_n\|$$ and $$\|\sum_{n=1}^m a_n g_{k_{q_n}}\|\approx \|\sum_{m=1}^n a_nf_n\|.$$   Then $$\frac{1}{r}\|\sum_{n=1}^m a_ng_{l_{p_{n-1}}} \|\leqslant \|\sum_{n=1}^m a_ng_{k_{q_n}} \|\leqslant r\|\sum_{n=1}^m a_n g_{l_{p_n}}\|.$$   From this it is easy to see that $(e_n)_{n=1}^\infty$ and $(f_n)_{n=1}^\infty$ are $r^2$-equivalent.

In \cite{KO}, the following theorem was shown. 

\begin{theorem} If every spreading model generated by a normalized, weakly null sequence in $X$ is dominated by the canonical $\ell_p$ basis, then there exists a constant $C$ such that every spreading model generated by a normalized, weakly null sequence in $X$ is $C$-dominated by the canonical $\ell_p$ basis. The same holds if we replace $\ell_p$ by $c_0$.

\label{KO2}
\end{theorem}

This theorem neither implies nor is implied by Theorem \ref{KO}.  A proof of Theorem \ref{KO2} was given in \cite{KO} to prepare the reader for the more technical proof of Theorem \ref{KO}.   We note that our proof of Theorem \ref{main} is essentially a transfinite version of the proof of Theorem \ref{KO2}, and therefore offers a unified approach to Theorems \ref{KO} and \ref{KO2}.  We make precise the connection between our proof of Theorem \ref{main} and our generalization of Theorem \ref{KO2}. In what follows, $\Gamma$ is still defined as in Section $3$.

\begin{theorem} Let $X$, $R$ be as in the introduction. Let $(g_n)_{n=1}^\infty$ be a normalized, bimonotone, right dominant basic sequence and let $(e_n)_{n=1}^\infty$ be a spreading model generated by a subsequence of $(g_n)_{n=1}^\infty$. The following are equivalent.  \begin{enumerate}[(i)]\item Every $*$-spreading model generated by a member of $R$ is dominated by $(e_n)_{n=1}^\infty$.  \item There exists a constant $C$ such that every $*$-spreading model generated by a member of $B_R$ is $C$-dominated by $(e_n)_{n=1}^\infty$. \item $\Gamma(\omega)<\infty$. \end{enumerate}

\label{main2}
\end{theorem}

We now give the proof of Theorem \ref{main2}. 

\begin{proof} We will show that $(ii)\Rightarrow (i)\Rightarrow (iii)\Rightarrow (ii)$

$(ii)\Rightarrow (i)$ This is clear.

$(i)\Rightarrow (iii)$ Assume every spreading model generated by a member of $R$ is dominated by $(e_n)_{n=1}^\infty$.  Recall that for some $q_n\uparrow \omega$, $$\mathcal{F}_\omega=\{F: \exists n\leqslant F\in \mathcal{F}_{q_n}\}.$$ Since $\mathcal{F}_{q_n}$ consists of sets with cardinality not more than $q_n$, it follows that $$\mathcal{F}_\omega=\{\varnothing\}\cup \{F: \varnothing\neq F, |F|\leqslant q_{\min F}\}.$$   

 In order to show $\Gamma(\omega)<\infty$, by Theorem \ref{good}, it suffices to show that for each $\varrho\in R$, $\Gamma(\omega, \varrho)<\infty$.   To that end, fix $\varrho=(x_n)_{n=1}^\infty\in R\subset \ell_\infty(X)$.   As noted above, a standard application of the finite Ramsey Theorem yields the existence of some $M\in[\nn]$ such that for each $m\in\nn$ and each scalar sequence $(a_n)_{n=1}^m$, the iterated limit $$\underset{l_1\to \infty}{\lim} \ldots \underset{l_m\to \infty}{\lim} \|\sum_{n=1}^m a_n x_{M(l_n)}\|$$ exists. We define a seminorm on $c_{00}$, whose natural basis we denote by $(f_n)_{n=1}^\infty$, by $$\|\sum_{n=1}^m a_nf_n\|= \underset{l_1\to \infty}{\lim} \ldots \underset{l_m\to \infty}{\lim} \|\sum_{n=1}^m a_n x_{M(l_n)}\|.$$  Therefore $(x_{M(n)})_{n=1}^\infty$ generates the $*$-spreading model $(f_n)_{n=1}^\infty$ with this seminorm.   By replacing $M$ with a subset thereof and using a finite net argument and a diagonalization argument, we may assume that for each $m\in\nn$, each $m\leqslant l_1<\ldots <l_{q_m}$, and each scalar sequence $(a_n)_{n=1}^{q_m} \in B_{\ell_\infty^{q_m}}$, $$\Bigl|\|\sum_{n=1}^{q_m} a_nf_n\|-\|\sum_{n=1}^{q_m} a_nx_{M(l_n)}\|\Bigr|<1.$$   Similarly, since  $(e_n)_{n=1}^\infty$ is generated by some subsequence of $(g_n)_{n=1}^\infty$, we may choose positive numbers $(\ee_n)_{n=1}^\infty$ and $L\in[\nn]$ such that for each $m\in [\nn]$, each $m\leqslant l_1<\ldots <l_{q_m}$, and each $(a_n)_{n=1}^{q_m} \in B_{\ell_\infty^{q_m}}$, $$\Bigl|\|\sum_{n=1}^{q_m} a_ne_n\|-\|\sum_{n=1}^{q_m} a_ng_{L(l_n)}\|\Bigr|<\ee_m.$$  Since the sequences $(e_n)_{n=1}^\infty$ and $(g_n)_{n=1}^\infty$ are normalized and  bimonotone, it follows that if $\ee_m$ was chosen small enough, then for each $m\in\nn$, $m\leqslant l_1<\ldots <l_{q_m}$, and each scalar sequence $(a_n)_{n=1}^{q_m}$, $$\|\sum_{n=1}^{q_m} a_ne_n\|\leqslant 2\|\sum_{n=1}^{q_m} a_ng_{L(l_n)}\|.$$  

Since $(i)$ is assumed to hold, there exists $C$ such that $(f_n)_{n=1}^\infty \leqslant_C (e_n)_{n=1}^\infty$.  We claim that $$\{(M(F), L(F)):F\in\mathcal{F}_\omega\}\subset T(\varrho, 1+2C),$$ which will show that $\Gamma(\omega, \varrho)<\infty$ and finish $(i)\Rightarrow (iii)$. To that end, fix $\varnothing \neq F\in \mathcal{F}_\omega$ and scalars $(a_n)_{n\in F}$. By extending $F$ to $F\cup G$ for some $F<G$ and letting $a_n=0$ for all $n\in G$ and relabeling if necessary, we can assume that $\min F=m$ and $|F|=q_m$.    Write $F=(l_1, \ldots, l_{q_m})$ and let $b_n=a_{l_n}$ for each $1\leqslant n\leqslant q_m$.     The case $a_n=0$ for all $n\in F$ is trivial, so assume $\max_{n\in F}|a_n|>0$. By scaling and using homogeneity, we can assume $\max_{n\in F}|a_n|=1$.   Since $(g_n)_{n=1}^\infty$ is normalized and bimonotone, $$1\leqslant \|\sum_{n\in F}a_n g_{L(n)}\|.$$   Moreover, by the preceding paragraph, $$\|\sum_{n\in F}a_n x_{M(n)}\|\leqslant 1+\|\sum_{n=1}^{q_m} b_nf_n\|\leqslant 1+C\|\sum_{n=1}^{q_m} b_ne_n\|\leqslant 1+2C\|\sum_{n\in F} a_n g_{L(n)}\|\leqslant (1+2C)\|\sum_{n\in F}a_n g_{L(n)}\|.$$  This gives the inclusion $$\{(M(F), L(F)):F\in\mathcal{F}_\omega\}\subset T(\varrho, 1+2C)$$ and finishes the proof.

$(iii)\Rightarrow (ii)$  Assume $\Gamma(\omega)<\infty$ and fix $\Gamma(\omega)<C<\infty$.   Recall that $\mathcal{F}_\omega$ is defined by first fixing some $q_n\uparrow \omega$ and then letting $$\mathcal{F}_\omega=\{F: \exists n\leqslant F\in \mathcal{F}_{q_n}\}.$$  This is equivalent to saying that a non-empty set $F$ lies in $\mathcal{F}_\omega$ if and only if $|F|\leqslant q_{\min F}$.    Now assume that  $\varrho=(x_n)_{n=1}^\infty\in B_R$  generates the $*$-spreading model $(f_n)_{n=1}^\infty$.   Fix $K,N\in[\nn]$ such that $$\{(K(F), N(F)): F\in\mathcal{F}_\omega\}\subset T(\varrho, C).$$   We may fix $s_1<s_2<\ldots$ such that $(g_{N(s_n)})_{n=1}^\infty$ generates some spreading model, $(e_n')_{n=1}^\infty$. Let $M(n)=K(s_n)$ and $L(n)=N(s_n)$. For any $F\in\mathcal{F}_\omega$, $G:=(s_n: n\in F)$ is a spread of $F$, and therefore also a member of $\mathcal{F}_\omega$. From this it follows that $(M(F), L(F))=(K(G), N(G))\in T(\varrho, C)$.   Fix scalars $(a_n)_{n=1}^m$. Note that for any $m\leqslant l_1<\ldots <l_m$, $(l_1, \ldots, l_m)\in \mathcal{F}_\omega$, since $$|(l_1, \ldots, l_m)|=m \leqslant q_m\leqslant q_{l_1}.$$  Note that $(x_{M(n)})_{n=1}^\infty$ also generates the $*$-spreading model $(f_n)_{n=1}^\infty$.  Therefore \begin{align*} \|\sum_{n=1}^m a_n f_n\| & = \underset{l_1\to \infty}{\lim} \ldots \underset{l_m\to \infty}{\lim} \|\sum_{n=1}^m a_n x_{M(l_n)}\| \\ & \leqslant C\text{\ } \underset{l_1\to \infty}{\lim} \ldots \underset{l_m\to \infty}{\lim} \|\sum_{n=1}^m a_n g_{K(l_n)}\| = C\|\sum_{n=1}^m a_ne_n'\|. \end{align*} 

Since $(e_n)_{n=1}^\infty$ and $(e_n')_{n=1}^\infty$ are $r$-equivalent, $(f_n)_{n=1}^\infty \leqslant_{rC} (e_n)_{n=1}^\infty$.

\end{proof}

Given $X$, $R$, and $(g_n)_{n=1}^\infty$ satisfying our usual conditions, Theorem \ref{main} is equivalent to the statement that if for each $\varrho\in R$, $\Gamma(\omega_1, \varrho)<\omega_1$, $\Gamma(\omega_1)<\infty$. Theorem \ref{ffold} deals with the condition $\Gamma(\omega)<\infty$.   Combined, these two theorems give four conditions: \begin{enumerate}\item For every $\varrho\in R$, $\Gamma(\omega, \varrho)<\infty$. \item For every $\varrho\in R$, $\Gamma(\omega_1, \varrho)<\infty$. \item $\Gamma(\omega)<\infty$. \item $\Gamma(\omega_1)<\infty$. \end{enumerate} It is the content of Theorems \ref{main} and \ref{ffold} that $(1)\Leftrightarrow (3)$ and $(2)\Leftrightarrow (4)$.  As noted in \cite{KO}, neither of Theorems \ref{KO}, \ref{KO2} implies the other. These do not imply each other, because one can choose $(g_n)_{n=1}^\infty$ in such a way that $\Gamma(\omega)<\infty$ and $\Gamma(\omega_1)=\infty$. 

It follows from Corollary  \ref{break} that if $\omega^\gamma\leqslant \xi<\omega^{\gamma+1}$, $\Gamma(\omega^\gamma)<\infty$ if and only if $\Gamma(\xi)<\infty$, so that $\Gamma(\omega^\gamma)<\infty$ and $\Gamma(\xi)<\infty$ are equivalent properties. The previous paragraph states that the property $\Gamma(\omega^1)<\infty$ is strictly weaker than the property $\Gamma(\omega^{\omega_1})<\infty$.   We may ask about intermediate values. For example, are the properties $\Gamma(\omega^1)<\infty$ and $\Gamma(\omega^2)<\infty$ equivalent? We know that the second implies the first by Lemma \ref{work}$(i)$. The next result shows that the answer to this question is no. That is, for each $1\leqslant \gamma_1, \gamma_2\leqslant \omega_1$ with $\gamma_1\neq \gamma_2$, the fact that $$\Gamma(\omega^{\gamma_1}, \varrho)<\infty\text{\ for all\ }\varrho\in R \Leftrightarrow \Gamma(\omega^{\gamma_1})<\infty$$ neither implies nor is implied by $$\Gamma(\omega^{\gamma_2}, \varrho)<\infty\text{\ for all\ }\varrho\in R \Leftrightarrow \Gamma(\omega^{\gamma_2})<\infty.$$

\begin{proposition} For any $1\leqslant \gamma<\omega_1$, there exist $X$, $R$, and $(g_n)_{n=1}^\infty$ such that $$\omega^\gamma=\min \{\xi: \Gamma(\xi)=\infty\}.$$

\label{sc}
\end{proposition}

\begin{proof} Fix $0\leqslant \nu<\omega_1$. Let $X=\ell_2$, $R=c_0^w(\ell_2)$, and let $(g_n)_{n=1}^\infty$ denote the canonical basis of the $2$-convexification $X^{(2)}_\nu$ of the Schreier space $X_\nu$. We recall that $$\|\sum_{n=1}^\infty a_ng_n\|_{X_\nu}=\sup_{F\in \mathcal{S}_\nu} \sum_{n\in F}|a_n|$$ and $$\|\sum_{n=1}^\infty a_ng_n\|_{X_\nu^{(2)}}=\sup_{F\in \mathcal{S}_\nu} \Bigl(\sum_{n\in F}|a_n|^2\Bigr)^{1/2}.$$

We note that since $\text{rank}(\mathcal{F}_{\omega^\nu})=\omega^\nu+1=\text{rank}(\mathcal{S}_\nu)$, \cite[Proposition $3.1$]{Con} yields the existence of some $P\in[\nn]$ such that for any $F\in\mathcal{F}_{\omega^\nu}$, $P(F)\in \mathcal{S}_\nu$. From this it follows that for any $F\in \mathcal{F}_{\omega^\nu}$, $$\|\sum_{n\in F}a_n g_{P(n)}\|_{X^{(2)}_\nu}= \bigl(\sum_{n\in F}|a_n|^2\bigr)^{1/2}.$$  For any normalized, weakly null sequence $(x_n)_{n=1}^\infty$ in $\ell_2$ and $\ee>0$, there exists $M\in[\nn]$ such that $(x_{M(n)})_{n=1}^\infty$ is $(1+\ee)$-equivalent to the canonical $\ell_2$ basis.  From this it follows that for any $F\in\mathcal{F}_{\omega^\nu}$ and scalars $(a_n)_{n\in F}$, $$\|\sum_{n\in F} a_nx_{M(n)}\|_{\ell_2} \leqslant (1+\ee)\bigl(\sum_{n\in F}|a_n|^2\bigr)^{1/2} =(1+\ee)\|\sum_{n\in F} a_ng_{P(n)}\|_{X_\nu^{(2)}}.$$  Therefore $\Gamma(\omega^\nu)\leqslant 1$. 

However, it is known (\cite{AMT}) that the basis $(g_n)_{n=1}^\infty$ is $\nu+1$-weakly null as a basis for $X_\nu$. This means that for any $Q\in[\nn]$, $$\inf \Bigl\{\|\sum_{n\in F} a_n g_{Q(n)}\|_{X_\nu}: F\in\mathcal{S}_{\nu+1}, \sum_{n\in F}|a_n|=1\Bigr\}=0.$$   From this it follows $\Gamma(\omega^{\nu+1})=\infty$. Indeed, let $\varrho=(x_n)_{n=1}^\infty$ denote the canonical $\ell_2$ basis and, again using \cite[Proposition $3.1$]{Con}, choose $P\in[\nn]$ such that $\{P(F): F\in \mathcal{S}_{\nu+1}\}\subset \mathcal{F}_{\omega^{\nu+1}}$.  Fix $M,L\in[\nn]$ and let $Q(n)=L(P(n))$ for each $n\in\nn$.   Note that there cannot exist $C$ such that $\{(M(F), L(F)): F\in \mathcal{F}_{\omega^{\nu+1}}\}\subset T(\varrho, C)$, since then \begin{align*} \frac{1}{C} & \leqslant \inf\Bigl\{\|\sum_{n \in F} a_ng_{L(n)}\|_{X^{(2)}_\nu} : F\in \mathcal{F}_{\omega^{\nu+1}}, \|\sum_{n\in F}a_nx_{M(n)}\|=1 \Bigr\} \\ & \leqslant \inf\Bigl\{\|\sum_{n \in P(F)} a_ng_{L(n)}\|_{X^{(2)}_\nu} : F\in \mathcal{S}_{\nu+1}, \sum_{n\in F}|a_n|^2=1 \Bigr\} \\ & = \inf\Bigl\{\|\sum_{n \in F} a_ng_{Q(n)}\|_{X^{(2)}_\nu} : F\in \mathcal{S}_{\nu+1}, \sum_{n\in F}|a_n|^2=1\Bigr\} \\ & = \inf\Bigl\{\|\sum_{n \in F} a_ng_{Q(n)}\|_{X_\nu}^{1/2} : F\in \mathcal{S}_{\nu+1}, \sum_{n\in F}|a_n|=1\Bigr\} \\ & =0.   \end{align*}

This yields that $\Gamma(\omega^{\nu+1})=\infty$. This gives the proposition in the case that $\gamma$ is a successor ordinal.

Now fix $1=\vartheta_1>\vartheta_2>\ldots$, $\lim_n \vartheta_n=0$, a limit ordinal $\gamma<\omega_1$, and $\nu_n\uparrow \gamma$.   Let $(g_n)_{n=1}^\infty$ be the basis for the completion $G$ of $c_{00}$ with respect to the norm $$\|\sum_{n=1}^\infty a_ng_n\|_G=\sup_k \vartheta_k \sup_{F\in \mathcal{S}_{\nu_k}} \sum_{n\in F}|a_n|.$$  Let $G^{(2)}$ denote the $2$-convexification of $G$, the norm of which is given by $$\|\sum_{n=1}^\infty a_ng_n\|_{G^{(2)}}=\sup_k \vartheta_k^{1/2} \sup_{F\in \mathcal{S}_{\nu_k}} \bigl(\sum_{n\in F}|a_n|^2\bigr)^{1/2}.$$  For $\nu<\gamma$, if $\nu<\nu_n$, again by \cite[Proposition $3.1$]{Con}, there exists $P\in[\nn]$ such that for any $F\in\mathcal{F}_{\omega^\nu}$, $P(F)\in \mathcal{S}_{\nu_n}$.  As in the preceding case, we deduce that $\Gamma(\omega^\nu)\leqslant 1/\theta_n^{1/2}$. It is known (see \cite{CN}) that the basis $(g_n)_{n=1}^\infty$ is $\gamma$-weakly null, and we deduce that $\Gamma(\omega^\gamma)=\infty$ as in the successor case.

\end{proof}

\section{Freeman's theorem}

Most of our examples of Theorem \ref{main} will be in the case that $R=c_0^w(X)$, the space of weakly null sequences in $X$,  with $\|\cdot\|_R=\|\cdot\|_{\ell_\infty(X)}$.    By homogeneity, the hypothesis that every member of $B_R$ has a subsequence dominated by a subsequence of $(g_n)_{n=1}^\infty$ is equivalent to the hypothesis that every normalized, weakly null sequence in $X$ has a subsequence dominated by a subsequence of $(g_n)_{n=1}^\infty$.   For the moment, we consider this particular choice $R=c_0^w(X)$.   We note that Theorem \ref{main} neither implies nor is implied by Theorem \ref{Freeman}. However, since the bases of $\ell_p$ and $c_0$ are equivalent to all their subsequences, Theorem \ref{Freeman} and Theorem \ref{main} are both generalizations of Theorem \ref{KO}.  In order to motivate Theorem \ref{main}, we provide several examples of spaces $X$ and non-trivial bases $(g_n)_{n=1}^\infty$ satisfying Theorem \ref{main}, but not satisfying Theorem \ref{Freeman} for any choice of $(v_n)_{n=1}^\infty$ except for the trivial case that $(v_n)_{n=1}^\infty$ is equivalent to the canonical $\ell_1$ basis.

We recall that for $1\leqslant \xi<\omega_1$, a sequence $(x_n)_{n=1}^\infty $ in some Banach space $X$ is called an $\ell_1^\xi$-\emph{spreading model} if $(x_n)_{n=1}^\infty$ is bounded and $$0< \inf\Bigl\{\|x\|: F\in \mathcal{S}_\xi, x=\sum_{n\in F}a_nx_n, \sum_{n\in F}|a_n|=1\Bigr\}.$$

Given a Banach space $X$ and a natural number $l$, let $$w_l(X)= \underset{(x_n)_{n=1}^\infty\in B_{c_0^w(X)}}{\sup} \inf\Bigl\{\|\sum_{n\in A} a_n x_n\|: A\subset \nn, |A|=l, \sum_{n\in A}|a_n|=1\Bigr\}.$$  Of course, $w_l(X)\leqslant 1$. By a standard James-type blocking argument, for all $l,m\in\nn$, $w_{lm}(X)\leqslant w_l(X)w_m(X)$.  From this and further standard arguments, it follows that either $w_l(X)=1$ for all $l\in\nn$ or for some $0<\delta$, $w_l(X)=O(l^{-\delta})$.  Furthermore, any Banach space $X$ which admits a weakly null $\ell_1^1$-spreading model satisfies $w_l(X)=1$ for all $l\in\nn$. However, the converse is not true. For example, the space $X_\text{OS}$ of Odell and Schlumprecht \cite{OS} which admits no $\ell_1^1$ spreading model is the completion of $c_{00}$ with respect to the implicitly defined norm  $$\|x\|_{\text{OS}}= \max \Bigl\{\|x\|_{c_0}, \Bigl(\sum_{i=1}^\infty \|x\|_{n_i}^2\Bigr)^{1/2}\Bigr\},$$ where $$\|x\|_{n_i}= \sup\Bigl\{\frac{1}{\log_2(n_i+1)} \sum_{j=1}^{n_i} \|I_jx\|_{\text{OS}}: I_1<\ldots <I_{n_i}\Bigr\}$$ and $(n_i)_{i=1}^\infty$ is a specifically chosen, very fast growing sequence of natural numbers.   Here $Ix$ is the projection of $x$ onto $\text{span}\{e_i: i\in I\}$.   This space $X_{\text{OS}}$ is reflexive and has no $\ell_1$ spreading model. However, it follows from the definition of the norm that for any $k\in\nn$ and any normalized block sequence  $(x_i)_{i=1}^{n_k}$ in $X_\text{OS}$, and any scalars $(a_i)_{i=1}^{n_k}$,  $$\|\sum_{i=1}^{n_k}a_ix_i\|_{\text{OS}} \geqslant \frac{\sum_{i=1}^{n_k} |a_i|}{\log_2(n_k+1)}.$$  Therefore $w_{n_i}(X_\text{OS})\geqslant \frac{1}{\log_2(n_i+1)}$ for all $i\in\nn$.  From this it follows that there cannot exist $0<\delta$ such that $w_l(X_\text{OS})=O(l^{-\delta})$, and $w_l(X_\text{OS})=1$ for all $l\in\nn$.

We now summarize the preceding discussion for later reference, which will be our tool for verifying that the spaces in our later examples only satisfy the hypothesis of Theorem \ref{Freeman} in the trivial case in which $(v_n)_{n=1}^\infty$ is equivalent to the canonical $\ell_1$ basis.

\begin{rem}\upshape Let $(v_n)_{n=1}^\infty$ be a seminormalized Schauder basis.  If $X$ is a Banach space in which $w_l(X)=1$ for all $l\in\nn$ and every normalized, weakly null sequence in $X$ has a subsequence dominated by $(v_n)_{n=1}^\infty$, then $(v_n)_{n=1}^\infty$ is equivalent to the canonical $\ell_1$ basis. Indeed, by Theorem \ref{Freeman}, there exists a constant $C$ such that any $(x_n)_{n=1}^\infty \in B_{c_0^w(X)}$ has a subsequence which is $C$-dominated by $(v_n)_{n=1}^\infty$.    But we can find for each $l\in\nn$ a sequence $(x_n)_{n=1}^\infty \in B_{c_0^w(X)}$ such that for any $A$ with $|A|=l$ and scalars $(a_n)_{n\in A}$ with $\sum_{n\in A}|a_n|=1$, $\|\sum_{n\in A} a_nx_n\|\geqslant 1/2$. Then for some subsequence $(y_n)_{n=1}^\infty$ of $(x_n)_{n=1}^\infty$ which is $C$-dominated by $(v_n)_{n=1}^\infty$ and any scalars $(a_n)_{n=1}^l$ with $\sum_{n=1}^l |a_n|=1$, $$\frac{1}{2C}\leqslant \frac{1}{C}\|\sum_{n=1}^l a_ny_n\| \leqslant \|\sum_{n=1}^l a_nv_n\|.$$  Since this holds for any $l\in\nn$, and since $(v_n)_{n=1}^\infty$ is bounded,  $(v_n)_{n=1}^\infty$ is equivalent to the canonical $\ell_1$ basis. 

\label{tri}

\end{rem}

\begin{example}\upshape For any regular family $\mathcal{F}$ containing all singletons, we define the space $X_\mathcal{F}$ to be the completion of $c_{00}$ with respect to the norm $$\|x\|_{X_\mathcal{F}}=\sup_{F\in \mathcal{F}} \sum_{n\in F}|x(n)|=\sup_{F\in \mathcal{F}}\|Fx\|_{\ell_1}.$$   The canonical $c_{00}$  basis is a normalized, $1$-unconditional, and $1$-right dominant basis for $X_\mathcal{F}$. The $1$-right dominance is a consequence of the spreading property of $\mathcal{F}$.  We also note that the canonical basis is weakly null, which can be seen by isomorphically embedding $X_\mathcal{F}$ into $C(\mathcal{F})$ as $f_x(E)=\sum_{n\in E} a_n$, where $x=\sum_{n=1}^\infty a_ne_n$. Then $x\mapsto f_x$ is an isomorphic embedding of $X_\mathcal{F}$ into $C(\mathcal{F})$ which sends the canonical $X_\mathcal{F}$ basis to a bounded, pointwise null sequence in $C(\mathcal{F})$. Therefore the canonical $X_\mathcal{F}$ basis is weakly null. If $\text{rank}(\mathcal{F})>\omega$, then some subsequence of the basis of $X_\mathcal{F}$ is an $\ell_1^1$-spreading model.  Indeed, if $\text{rank}(\mathcal{F})>\omega$, then for any $l\in\nn$, $\mathcal{F}$ must contain some $F$ with $|F|\geqslant l$. But then by the spreading and hereditary properties of $\mathcal{F}$, $\mathcal{F}$ must contain every $l$-element set whose minimum is at least $\max F$, since any such set is  a subset of a spread of $F$. Therefore there exist $n_1<n_2<\ldots$ such that if $n_l\leqslant F$ and $|F|\leqslant l$, then $F\in \mathcal{F}$.   From this it follows that $(e_{n_l})_{l=1}^\infty$ is an isometric $\ell_1^1$-spreading model.  As noted in Remark \ref{tri}, $X_\mathcal{F}$ cannot satisfy any non-trivial version of Theorem \ref{Freeman}.

However, if $(x_n)_{n=1}^\infty$ is any normalized block sequence in $X_\mathcal{F}$, then $$(x_n)_{n=1}^\infty \leqslant_1 (e_{\max \text{supp}(x_n)})_{n=1}^\infty.$$  Indeed, fix scalars $(a_n)_{n=1}^\infty \in c_{00}$ and $F\in \mathcal{F}$ such that $$\|\sum_{n=1}^\infty a_nx_n\|_{X_\mathcal{F}}=\|F\sum_{n=1}^\infty a_nx_n\|_{\ell_1}.$$  Now if $A=\{n: Fx_n\neq 0\}$ and for each $n\in A$, we fix $p_n\in F\cap \text{supp}(x_n)$, then $G:=(\max \text{supp}(x_n))_{n\in A}$ is a spread of $(p_n)_{n\in A}\subset F$. Furthermore, $$\|F\sum_{n=1}^\infty a_nx_n\|_{\ell_1}\leqslant \sum_{n\in F}|a_n|= \|G\sum_{n=1}^\infty a_ne_{\max \text{supp}(x_n)}\|_{\ell_1}\leqslant \|\sum_{n=1}^\infty a_ne_{\max \text{supp}(x_n)}\|_{X_\mathcal{F}}.$$   

From this and a standard perturbation argument, for any $\ee>0$ and any normalized, weakly null sequence $(x_n)_{n=1}^\infty$ in $X_\mathcal{F}$, there exists a subsequence of $(x_n)_{n=1}^\infty$ which is $1+\ee$-dominated by a subsequence of the canonical basis of $X_\mathcal{F}$. 

\label{schreier}
\end{example}

\begin{example}\upshape Let $\mathcal{F}$, $\mathcal{G}$ be regular families containing all singletons and let $\xi, \zeta<\omega_1$ be such that  $\text{rank}(\mathcal{F})=\xi+1$ and $\text{rank}(\mathcal{G})=\zeta+1$. Assume also that  $\omega\leqslant \xi< \zeta \omega$. From this it follows that there exists $n\in\nn$ such that $\xi\leqslant \zeta n$.   Let $$\mathcal{H}= \{\varnothing\}\cup \Bigl\{\bigcup_{i=1}^t F_i: F_1<\ldots <F_t, \varnothing\neq F_i\in \mathcal{G}, t\leqslant n\Bigr\}.$$  Then $\text{rank}(\mathcal{H})=\zeta n+1\geqslant \text{rank}(\mathcal{F})$.    Now by \cite[Proposition $3.1$]{Con}, there exists $M\in[\nn]$ such that for any $E\in \mathcal{F}$, $M(E)\in \mathcal{H}$.  Therefore the basis $(e_i)_{i=1}^\infty$ of $X_\mathcal{F}$ is $1$-dominated by $(e_{M(i)})_{i=1}^\infty\subset X_\mathcal{H}$, which is $n$-dominated by $(e_{M(i)})_{i=1}^\infty \subset X_\mathcal{G}$.  Combining this with the previous example, any normalized, weakly null sequence in $X_\mathcal{F}$ has a subsequence which is $n+\ee$-dominated by a subsequence of the $X_\mathcal{G}$ basis.  

Taking $\mathcal{G}=\mathcal{S}_\xi$ for some $1\leqslant\xi<\omega_1$ and $\mathcal{F}=\mathcal{H}$ as defined in the preceding paragraph, and using known facts concerning the repeated averages hierarchy, the constant of $n$ as the infimum of $C$ such that every normalized, weakly null sequence in $X_\mathcal{F}$ is $C$-dominated by a subsequence of the $X_\mathcal{G}$ basis is sharp. 

\label{schreier2}
\end{example}

We recall that for $\xi<\omega_1$,  $X_\xi:=X_{\mathcal{S}_\xi}$ was already introduced in the proof of Proposition \ref{sc}. The spaces $X_\xi$, $\xi<\omega_1$, are the \emph{Schreier spaces}.

\begin{example} \upshape For $0\leqslant \xi<\omega_1$ and $1<p<\infty$, the Baernstein space $X_{\xi,p}$ is the completion of $c_{00}$ with respect to the norm $$\|x\|_{\xi,p}=\sup \Bigl\{\Bigl(\sum_{n=1}^\infty \|F_nx\|_{\ell_1}^p\Bigr)^{1/p}: F_1<F_2<\ldots, F_n\in \mathcal{S}_\xi\Bigr\}.$$ The spreading property of $\mathcal{S}_\xi$ and subsymmetry of $\ell_p$ yield that the canonical basis of $X_{\xi,p}$ is $1$-right dominant. For $0<\xi$, the basis of $X_{\xi,p}$ is a weakly null $\ell_1^1$-spreading model, and so $X_{\xi,p}$ cannot satisfy the hypotheses of Theorem \ref{Freeman} except in the trivial case in which $(v_n)_{n=1}^\infty$ is equivalent to the canonical $\ell_1$ basis.   It was shown in \cite{C2} that for any normalized block sequence $(x_n)_{n=1}^\infty$ in $X_{\xi,p}$, $(x_n)_{n=1}^\infty \leqslant_4 (g_{\max \text{supp}(x_n)})_{n=1}^\infty$. Therefore any normalized, weakly null sequence in $X_{\xi,p}$ has a subsequence $4+\ee$-dominated by a subsequence of the $X_{\xi,p}$ basis.

\end{example}

\begin{definition} Let us say a normalized basis $(e_n)_{n=1}^\infty$ for a Banach space $T$ is \emph{block stable} if for any integers  $0=m_0<m_1<\ldots$  and any normalized block sequence $(x_n)_{n=1}^\infty$ in $T$ such that  $\text{supp}(x_n)\subset (m_{n-1}, m_n]$, then $(x_n)_{n=1}^\infty$ is equivalent to $(e_{m_n})_{n=1}^\infty$. A standard gliding hump argument shows that if $(e_n)_{n=1}^\infty$ is a block stable basis for $T$, then there exists a constant $B$ such that for any integers  $0=m_0<m_1<\ldots$  and any normalized block sequence $(x_n)_{n=1}^\infty$ in $T$ such that  $\text{supp}(x_n)\subset (m_{n-1}, m_n]$,  $(x_n)_{n=1}^\infty$ is $B$-equivalent to $(e_{m_n})_{n=1}^\infty$. Clearly any normalized, weakly null sequence in $T$ has a subsequence $B+\ee$-dominated by a subsequence of the basis $(e_n)_{n=1}^\infty$. 

Formally speaking, block stability is a property of a given basis of the space $T$, and not of the space $T$ itself. However, if $T$ has a canonical basis, then we shall say $T$ is block stable if its canonical basis is. 

\end{definition}

\begin{rem}\upshape None of our preceding examples has been block stable.  This is because when $\text{rank}(\mathcal{F})>\omega$, $X_\mathcal{F}$ admits a copy of $c_0$, but the basis admits no subsequence equivalent to the canonical $c_0$ basis. Similarly,  for $0<\xi<\omega_1$, $X_{\xi,p}$ admits a copy of $\ell_p$, but the basis admits no subsequence equivalent to the canonical $\ell_p$ basis. Therefore when $\text{rank}(\mathcal{F})>\omega$ and $0<\xi<\omega_1$, every normalized block sequence in either $X_\mathcal{F}$ or $X_{\xi,p}$  admits a further block sequence which is dominated by some, but not equivalent to any,  subsequence of the basis.

\end{rem}

\begin{example}\upshape For $1\leqslant \xi<\omega_1$ and $0<\vartheta<1$, the Tsirelson space $T_{\xi,\vartheta}$ is the completion of $c_{00}$ with respect to the implicitly defined norm $$\|x\|=\max\Bigl\{\|x\|_{c_0}, \vartheta \sup\bigl\{\sum_{n=1}^t \|I_n x\|:  I_1<\ldots <I_t, (\min I_n)_{n=1}^t\in \mathcal{S}_\xi\bigr\}\Bigr\}.$$  The spreading property of $\mathcal{S}_\xi$ yields that the basis of $T_{\xi,\vartheta}$ is $1$-right dominant.  Furthermore, it was shown in \cite{CJT} that $T_{1,\frac{1}{2}}$ is block stable, and in \cite{LT} it was shown that $T_{\xi, \vartheta}$ for any $1\leqslant \xi<\omega_1$ and $0<\vartheta<1$. An easy duality argument yields that $T_{\xi,\vartheta}^*$ is also block stable. It is easy to see that every normalized block sequence in $T_{\xi, \vartheta}$ is an $\ell_1^1$-spreading model. Indeed, if $(x_n)_{n=1}^\infty$ is a normalized block sequence in $T_{\xi, \vartheta}$ and $I_1<I_2<\ldots$ are intervals such that $\text{supp}(x_n)\subset I_n$, and if $F\in\mathcal{S}_1\subset \mathcal{S}_\xi$, then for any scalars $(a_n)_{n\in F}$, $$\|\sum_{n\in F}a_nx_n\|\geqslant \vartheta \sum_{i\in F}\|I_i \sum_{n\in F} a_n\|\geqslant \vartheta\sum_{n\in F} \|a_nx_n\|=\vartheta\sum_{n\in F}|a_n|.$$     Now block stability implies that every normalized block sequence in $T_{\xi, \vartheta}$ is dominated by (and actually equivalent to) a subsequence of the basis.  Thus the space $X=T_{\xi, \vartheta}$ for $0<\xi<\omega_1$ and $0<\vartheta<1$ with the basis $(g_n)_{n=1}^\infty=(e_n)_{n=1}^\infty$ satisfy Theorem \ref{main} but do not satisfy Theorem \ref{Freeman} for any non-trivial $(v_n)_{n=1}^\infty$. 

We also note that for any $1<p<\infty$, the canonical basis of the $p$-convexification $T^{(p)}_{\xi,\vartheta}$ has the properties of $1$-right dominance and  block stability, from which it follows that  any normalized block sequence in $T_{\xi,\vartheta}^{(p)}$ is dominated by a subsequence of $(e_n)_{n=1}^\infty$. In this case, every normalized block sequence in $T^{(p)}_{\xi,\vartheta}$ is dominated by the $\ell_p$ basis with constant $1$. However, since every subsequence of the $T^{(p)}_{\xi,\vartheta}$ basis is $1$-dominated by and not equivalent to the $\ell_p$ basis, the property that every normalized, weakly null sequence in $T^{(p)}_{\xi, \vartheta}$ has a subsequence dominated by some subsequence of the basis of $T^{(p)}_{\xi, \vartheta}$ is a strictly stronger property than the property that every normalized, weakly null sequence in $T^{(p)}_{\xi, \vartheta}$ has a subsequence dominated by the canonical $\ell_p$ basis.  One can see with this example that, if $(v_n)_{n=1}^\infty$ is any seminormalized basis such that any normalized, weakly null sequence in $T_{\xi, \vartheta}^{(p)}$ has a subsequence dominated by $(v_n)_{n=1}^\infty$ (that is, if $T_{\xi, \vartheta}^{(p)}$ and $(v_n)_{n=1}^\infty$ satisfy the hypotheses of Freeman's theorem), then there exists $C$ such that $\|\sum_{n=1}^\infty a_nv_n\|\geqslant \bigl(\sum_{n=1}^\infty |a_n|^p\bigr)^{1/p}/C$ for all $(a_n)_{n=1}^\infty\in c_{00}$. 

We note that every normalized block sequence in $T_{\xi, \vartheta}^*$ is also dominated by a subsequence of the basis of $T^*_{\xi, \vartheta}$. Since the basis of $T_{\xi, \vartheta}^*$ is $1$-left dominant, any sequence dominated by a subsequence of the $T_{\xi, \vartheta}^*$ basis is also dominated by the $T_{\xi, \vartheta}^*$ basis. The spaces $T_{\xi, \vartheta}^*$ were among the examples given by Freeman to show that Theorem \ref{Freeman} is a genuine extension of Theorem \ref{KO}.

\end{example}

\begin{example}\upshape It was shown in \cite{CN} that if $X$ is a separable Banach space, then the Szlenk index of $X$ fails to exceed $\omega^\xi$ if and only if there exists a Banach space $G$ having Szlenk index not exceeding $\omega^\xi$ and having a normalized, $1$-unconditional, $1$-right dominant basis $(g_n)_{n=1}^\infty$, a constant $C$, a Markushevich basis $(x_n)_{n=1}^\infty$ for $X$, and integers $0=k_0<k_1<\ldots$ such that, with $E_n=\text{span}\{x_i: k_{n-1}<i\leqslant k_n\}$ such that, for any $0=r_0<r_1<\ldots$ and $(u_n)_{n=1}^\infty \in B_X^\nn\cap \prod_{n=1}^\infty \text{span}\{F_i: r_{n-1}<i\leqslant r_n\}$, $(u_n)_{n=1}^\infty\leqslant_C (g_{r_n})_{n=1}^\infty$.     From this it follows that if $X$ is a separable Banach space $X$ with Szlenk index $\omega^\xi$, then with $G$ as above, every normalized, weakly null sequence in $X$ has a subsequence dominated by a subsequence of $(g_n)_{n=1}^\infty$.    If $\xi=1$, then $X$, being separable with Szlenk index $\omega$, must be $p$-asymptotically uniformly smoothable for some $1<p<\infty$, and every normalized, weakly null sequence in $X$ has a subsequence dominated by the $\ell_p$ basis. This yields that in the $\xi=1$ case, such a space $X$ must satisfy the hypothesis of Theorem \ref{Freeman}, and in fact Theorem \ref{KO}, for the basis of some $\ell_p$ space with $1<p<\infty$.  However, for $1<\xi<\omega_1$, and $\xi\notin \{\omega^\eta: \eta\text{\ is a limit ordinal}\}$, there is a Banach space $X$ with Szlenk index $\omega^\xi$ admitting a weakly null $\ell_1^1$-spreaidng model. Therefore for some normalized, bimonotone, $1$-right dominant basis $(g_n)_{n=1}^\infty$ for a space $G$ having Szlenk $\omega^\xi$, $R=c_0^w(X)$ and $(g_n)_{n=1}^\infty$ satisfy the hypotheses of Theorem \ref{main}, but $X$ cannot satisfy the hypotheses of Theorem \ref{main} for any non-trivial basis $(v_n)_{n=1}^\infty$.

\label{Szlenk}
\end{example}

\begin{definition} Given a Banach space $X$, an ordinal $0<\xi<\omega_1$, and a weakly null sequence $(x_n)_{n=1}^\infty$ in $X$, we say $(x_n)_{n=1}^\infty$ is $\xi$-\emph{weakly null} if it does not have a subsequence which is an $\ell_1^\xi$-spreading model.  Given a sequence $(x_n)_{n=1}^\infty$ and $\ee>0$, we let $$\mathfrak{F}_\ee((x_n)_{n=1}^\infty)= \Bigl\{F\in[\nn]^{<\omega}:(\exists x^*\in B_{X^*})(\forall n\in F)(|x^*(x_n)|\geqslant \ee)\Bigr\}.$$   We note that, as was shown in \cite{CN}, a given weakly null sequence $(x_n)_{n=1}^\infty$ is $\xi$-weakly null if and only if for any $M\in[\nn]$ and $\ee>0$, there exists $N\in [M]$ such that the Cantor-Bendixson index of $\mathfrak{F}_\ee((x_n)_{n=1}^\infty) \upp N$ is less than $\omega^\xi$. 

\label{st}

\end{definition}

\begin{proposition} For $1\leqslant \xi<\omega_1$ and $\xi$-weakly null sequence $(x_n)_{n=1}^\infty\subset B_X$ in the Banach space $X$, for any $\ee>0$,  $(x_n)_{n=1}^\infty$ admits a subsequence which is $1+\ee$-dominated by a subsequence of the canonical basis $(g_n)_{n=1}^\infty$ of the Schreier space $X_\xi$.

\label{wn}
\end{proposition}

\begin{proof} Fix $\ee>0$ and  $0<\phi<1$ such that $(1-\phi)^2(1+\ee)>1$.      Using the note at the end of the definition of $\xi$-weakly null, we may apply Theorem \ref{gasp} to recursively select $M_1\supset M_2\supset \ldots$ such that for each $k\in\nn$, either $\mathcal{S}_\xi\upp M_k\subset \mathfrak{F}_{\phi^k}((x_n)_{n=1}^\infty)$ or $\mathfrak{F}_{\phi^k}((x_n)_{n=1}^\infty)\upp M_k\subset \mathcal{S}_\xi$.   Note that  the Cantor-Bendixson index of $\mathcal{S}_\xi\upp N$ is $\omega^\xi+1$ for all $N\in [\nn]$, from which it follows that for each $k\in\nn$, $\mathfrak{F}_{\phi^k}((x_n)_{n=1}^\infty)\upp M_k \subset \mathcal{S}_\xi$. Otherwise for every $N\in [M_k]$, the Cantor-Bendixson index of $\mathfrak{F}_{\phi^k}((x_n)_{n=1}^\infty)\upp N$ would exceed $\omega^\xi$, contradicting $\xi$-weak nullity of $(x_n)_{n=1}^\infty$.  

Fix $M(1)<M(2)<\ldots$ with $M(k)\in M_k$.    For any $(a_n)_{n=1}^\infty\in c_{00}$, we can fix $x^*\in B_{X^*}$ such that $$\|\sum_{n=1}^\infty a_n x_{M(n)}\|=\text{Re\ }x^*\Bigl(\sum_{n=1}^\infty a_nx_{M(n)}\Bigr).$$   For each $k\in\nn$, let $$A_k=\{k<n: |x^*(x_{M(n)})|\in (\phi^k, \phi^{k-1}]\}$$ and $$B_k=\{k\geqslant n: |x^*(x_{M(n)})|\in (\phi^k, \phi^{k-1}]\}.$$  By our choices above, $(M(n))_{n\in B_k}\in \mathcal{S}_\xi$.     Using the fact that the basis of $X_\xi$ is normalized and $1$-unconditional, \begin{align*} \|\sum_{n=1}^\infty a_n x_{M(n)}\| & = \text{Re\ }x^*\Bigl(\sum_{n=1}^\infty a_nx_{M(n)}\Bigr) \leqslant \sum_{k=1}^\infty \phi^{k-1}\sum_{n\in A_k\cup B_k} |a_n| \\ & \leqslant \sum_{k=1}^\infty \phi^{k-1}\Bigl(|A_k|+\|\sum_{n\in B_k}a_n g_{M(n)}\|_{X_\xi}\Bigr) \\ & \leqslant \sum_{k=1}^\infty \|\sum_{n=1}^\infty a_ng_{M(n)}\|_{X_\xi}  k\phi^{k-1}=\frac{1}{(1-\phi)^2} \|\sum_{n=1}^\infty a_n g_{M(n)}\|_{X_\xi} \\ & \leqslant (1+\ee) \|\sum_{n=1}^\infty a_n g_{M(n)}\|_{X_\xi}. \end{align*}

\end{proof}

\begin{example}\upshape For $1\leqslant \xi<\omega_1$, any Banach space $X$ which admits no weakly null $\ell_1^\xi$-spreading model (that is, any Banach space with the $\xi$-\emph{weak Banach-Saks property}), has the property that  every normalized, weakly null sequence in $X$ has a subsequence $1+\ee$-dominated by the basis of $X_\xi$.  This is an immediate consequence of the preceding proposition, once we note that in any such space, any normalized, weakly null sequence is $\xi$-weakly null.  The Odell-Schlumprecht space $X_\text{OS}$ is $1$-weak Banach-Saks, yielding an example of a space having the property that every normalized, weakly null sequence in $X_\text{OS}$ has a subsequence $1+\ee$-dominated by a subsequence of the $X_1$ basis, but to which Theorem \ref{Freeman} does not apply for any non-trivial $(v_n)_{n=1}^\infty$.

\label{Schreier again}
\end{example}

Now for a Banach space $X$ and $0<\xi<\omega_1$, we let $c_0^\xi(X)$ denote the space of $\xi$-weakly null sequences in $X$. Endowed with the $\ell_\infty(X)$ norm, $c_0^\xi(X)$ is a closed subspace of $c_0^w(X)$, and is therefore a closed subspace of $\ell_\infty(X)$.  To see that $c_0^\xi(X)$ is closed, we simply observe that its complement in $c_0^w(X)$ is open.  To see this, note that if $(x_n)_{n=1}^\infty\in c_0^w(X)\setminus c_0^\xi(X)$, there exist $\ee>0$ and $M\in[\nn]$ such that $$2\ee\leqslant \inf\Bigl\{\|\sum_{n\in F} a_n x_{M(n)}\|: F\in\mathcal{S}_\xi, \sum_{n\in F}|a_n|=1\Bigr\}.$$ By the triangle inequality,  if $(y_n)_{n=1}^\infty \in c_0^w(X)$ is such that  $\|(y_n)_{n=1}^\infty - (x_n)_{n=1}^\infty\|_{\ell_\infty(X)}<\ee$, then $$\ee\leqslant \inf\Bigl\{\|\sum_{n\in F} a_n y_{M(n)}\|: F\in\mathcal{S}_\xi, \sum_{n\in F}|a_n|=1\Bigr\},$$ and $(y_n)_{n=1}^\infty\in c_0^w(X)\setminus c_0^\xi(X)$. 

It follows immediately from Definition \ref{st} that if $R=c_0^\xi(X)$ and $\|\cdot\|_R=\|\cdot\|_{\ell_\infty(X)}$, then $R$ satisfies all of the requirements stated in the introduction required for Theorem \ref{main} to apply.

\begin{example}\upshape The details of this example can be found in \cite{C}, wherein a detailed study of weak nullity of sequences in $X_\xi$ was undertaken.    Fix $0<\xi<\omega_1$ and let $\xi=\omega^{\ee_0}+\omega^{\ee_1}+\ldots +\omega^{\ee_{l-1}}$, where $l\in\nn$ and $\ee_0\geqslant \ldots \geqslant \ee_{l-1}$. Note that if $l=1$, $\xi=\omega^{\ee_0}$.  It is known, by utilizing the Cantor normal form of $\xi$, that $\xi$ admits such a representation. Furthermore, this representation is unique. For $0\leqslant i\leqslant l$, let $$\lambda_i= \omega^{\ee_0}+\ldots +\omega^{\ee_{i-1}}$$ and $$\rho_i=\omega^{\ee_i}+\ldots +\omega^{\ee_{l-1}},$$ where $\lambda_0=\rho_l=0$ by convention.  Note that for each $0\leqslant i\leqslant l$, $\lambda_i+\rho_i=\xi$.    In particular, $\lambda_l=\rho_0=\xi$. 

It was shown in \cite{C} that if $(x_n)_{n=1}^\infty$ is any seminormalized, weakly null sequence in $X_\xi$, then there exists a unique $0\leqslant i\leqslant l$ such that $(x_n)_{n=1}^\infty$ is $\rho_i+1$-weakly null and admits a subsequence equivalent to some subsequence of the canonical basis of $X_{\rho_i}$. Since the basis of $X_{\rho_i}$, and all of its subseqences, are $\ell_1^{\rho_i}$ spreading models, the latter condition means that the sequence $(x_n)_{n=1}^\infty$ is not $\rho_i$-weakly null. This yields that for any seminormalized, weakly null sequence  $(x_n)_{n=1}^\infty$ in $X_\xi$, there exists a unique $0\leqslant i\leqslant l$ such that $(x_n)_{n=1}^\infty$ is $\rho_i+1$-weakly null and not $\rho_i$-weakly null.  Furthermore, for each $0\leqslant i\leqslant l$, there exists a normalized, weakly null sequence $(x_n)_{n=1}^\infty$ in $X_\xi$ equivalent to a subsequence of the $X_{\rho_i}$ basis.  From this it follows that the hierarchy of weakly null sequences in $X_\xi$ is completely given in the list \begin{align*} c_0(X_\xi) & \subsetneq c_0^1(X)= c_0^{\rho_l+1}(X_\xi) = c_0^{\rho_{l-1}}(X_\xi) \subsetneq c_0^{\rho_{l-1}+1}(X_\xi)=c_0^{\rho_{l-2}}(X_\xi)\subsetneq \ldots  \\ &  \subsetneq c_0^{\rho_1+1}(X_\xi)= c_0^\xi(X_\xi)\subsetneq c_0^{\rho_0+1}(X_\xi)=c_0^{\xi+1}(X_\xi)=c_0^w(X_\xi).\end{align*} 

For $0\leqslant i\leqslant l$, let  $(g^i_n)_{n=1}^\infty$ denote the canonical basis of $X_{\rho_i}$. Then $(g^i_n)_{n=1}^\infty$ is normalized, $1$-unconditional, and $1$-right dominant.   Furthermore, with $R=c_0^{\rho_i+1}(X_\xi)$, every member of $R$ admits a subsequence dominated by a subsequence of $(g^i_n)_{n=1}^\infty$. Therefore there exists a constant $C$ such that any member of $B_{c_0^{\rho_i+1}(X_\xi)}(X_\xi)$ has a subsequence $C$-dominated by a subsequence of $(g^i_n)_{n=1}^\infty$. Furthermore, any member $(x_n)_{n=1}^\infty$ of $c_0^{\rho_i+1}(X_\xi)\setminus c_0^{\rho_i}(X_\xi)$ has a subsequence equivalent to the canonical $X_{\rho_i}$ basis. If $i<l$, $(x_n)_{n=1}^\infty$ has a subsequence which is an $\ell_1^1$-spreading model, and therefore cannot satisfy the hypotheses of Theorem \ref{Freeman} except in the trivial case in which $(v_n)_{n=1}^\infty$ is equivalent to the canonical $\ell_1$ basis. Still under the condition $i<l$, $(x_n)_{n=1}^\infty \in c_0^{\rho_i+1}(X_\xi)\setminus c_0^{\rho_i}(X_\xi)$ has a subsequence equivalent to a subsequence of $(g^i_n)_{n=1}^\infty$, and this subsequence does not admit any further subsequence dominated by a subsequence of $(g^j_n)_{n=1}^\infty$ for any $i<j\leqslant l$.

\label{last2}
\end{example}


\begin{thebibliography}{HD}

\normalsize
\baselineskip=17pt

\bibitem{AMT}  S. A. Argyros, S. Mercourakis,  A. Tsarpalias, \emph{ Convex unconditionality and summability
of weakly null sequences}, Israel J. Math. \textbf{107}:157-193, 1998.


\bibitem{C2} R.M. Causey, \emph{Estimation of the Szlenk index of reflexive Banach spaces using generalized Baernstein spaces}, Fund. Math., 228:153-171, 2015. 


\bibitem{Con} R.M. Causey, \emph{Concerning the Szlenk index}, Studia Math., 236:201-244, 2017. 


\bibitem{C} R.M. Causey, \emph{The $\xi,\zeta$-Dunford-Pettis Property}, submitted. 

\bibitem{CN} R.M. Causey, K. Navoyan, \emph{Factorization of Asplund operators}, J. Math. Anal. Appl., DOI 10.1016/j.jmaa.2019.06.081. 

\bibitem{CJT}  P. G. Casazza, W. B. Johnson, and L. Tzafriri, \emph{ On Tsirelson's space},  Israel J. Math., 47(2-3):81-98, 1984.

\bibitem{F} D. Freeman, \emph{Weakly null sequences with upper estimates}, Studia Math., 184(1): 79-102, 2008. 

\bibitem{G} I. Gasparis, \emph{A dichotomy theorem for subsets of the power subsets of
the power set of the natural numbers}, Proc. Amer. Math. Soc., 129:759-764, 2001.


\bibitem{KO2} H. Knaust and E. Odell, \emph{On $c_0$-sequences in Banach spaces}, Isreal J. Math. 67:153-169, 1989.

\bibitem{KO} H. Knaust and E. Odell, \emph{Weakly null sequences with upper $\ell_p$ estimates}, pp.85-107, in:
``Functional Analysis, Proceedings, The University of Texas at Austin, 1987-89,'' E. Odell,
H. Rosenthal (eds.), Lecture Notes in Mathematics 1470, Springer-Verlag, Berlin 1991.

\bibitem{LT}  D. H. Leung and W.K. Tang, \emph{The Bourgain $\ell_1$-index of mixed Tsirelson space},  J.
Funct. Anal., 199(2):301-331, 2003.

\bibitem{OS}  E. Odell, T. Schlumprecht, \emph{On the richness of the set of $p$'s in Krivine's theorem}, Geometric aspects of functional analysis (Israel, 1992-1994), volume 77 of \emph{Oper. Theory Adv. Appl.}, 177-198. Birkh\"{a}user, Basel, (1995).


\bibitem{T} B. S. Tsirelson, \emph{Not every Banach space contains $\ell_p$  or $c_0$}, Funct. Anal. Appl., 8:138-141,
1974




\end{thebibliography}
\end{document}